\documentclass[10pt]{amsart}
\usepackage{amsmath,amssymb,amscd,amsthm,amsfonts,amstext,amsbsy,mathrsfs,hyperref,upgreek,mathtools,stmaryrd,enumitem,bbm}

\usepackage[textsize=footnotesize,color=blue!40, bordercolor=white]{todonotes}
\usepackage{verbatim}
\usepackage%[scale=0.682]
[a4paper, margin=1.2in]{geometry}

\newcommand\V{\mathsf{V}}

\renewcommand{\L}{\mathcal{L}}
\newcommand{\C}{\mathcal{C}}

\newcommand{\E}{\mathcal{E}}

\newcommand{\PP}{\mathbb{P}}
\newcommand{\QQ}{\mathbb{Q}}

\newcommand{\BB}{\mathbb{B}}
\newcommand{\MM}{\mathbb{M}}

\newcommand{\FF}{\mathbb{F}}

\newcommand{\anf}[1]{{\text{``}\hspace{0.3ex}{#1}\hspace{0.01ex}\text{''}}}
\newcommand{\op}{\mathsf{op}}
\newcommand{\rnk}{\mathsf{rnk}}
\newcommand{\crnk}[2]{\mathsf{rnk}_{#1}(#2)}

\newcommand{\range}{\operatorname{range}}
\newcommand{\rank}{\operatorname{rank}}
\newcommand{\Col}{\operatorname{Col}}
\newcommand{\ra}{\rightarrow}

\newcommand{\Llr}{\Longleftrightarrow}
\newcommand{\tc}{\operatorname{tc}}
\newcommand{\tn}{\operatorname{tn}}

\newcommand{\one}{\mathbbm1}

\newcommand{\gbr}[1]{{\ulcorner{#1}\urcorner}}

\newcommand{\On}{{\mathrm{Ord}}}
\newcommand{\forces}{\Vdash}

\newcommand{\Fm}{{\mathrm{Fml}}}
\newcommand{\ZFC}{{\sf ZFC}}
\newcommand{\ZF}{{\sf ZF}}

\newcommand{\KM}{{\sf KM}}
\newcommand{\GB}{{\sf GB}}
\newcommand{\GBC}{{\sf GBC}}
\newcommand{\CH}{{\sf CH}}

\newcommand{\Set}[2]{\{{#1}~\vert~{#2}\}}
\newcommand{\ran}[1]{{{\rm{ran}}(#1)}}
\newcommand{\dom}[1]{{{\rm{dom}}(#1)}}

\newtheorem{theorem}{Theorem}[section]
\newtheorem{lemma}[theorem]{Lemma}
\newtheorem{corollary}[theorem]{Corollary}

\newtheorem{question}[theorem]{Question}

\newtheorem{observation}[theorem]{Observation}
\newtheorem{claim}{Claim}
\newtheorem*{claim*}{Claim}
\newtheorem*{subclaim*}{Subclaim}

\theoremstyle{definition}
\newtheorem{definition}[theorem]{Definition}

\newtheorem{example}[theorem]{Example}
\theoremstyle{remark}
\newtheorem{remark}[theorem]{Remark}
\newtheorem*{notation}{Notation}

\newenvironment{enumerate-(a)}{\begin{enumerate}[label={\upshape (\alph*)}, leftmargin=2pc]}{\end{enumerate}}

\newenvironment{enumerate-(a)-r}{\begin{enumerate}[label={\upshape (\alph*)}, leftmargin=2pc,resume]}{\end{enumerate}}

\newenvironment{enumerate-(A)}{\begin{enumerate}[label={\upshape (\Alph*)}, leftmargin=2pc]}{\end{enumerate}}

\newenvironment{enumerate-(A)-r}{\begin{enumerate}[label={\upshape (\Alph*)}, leftmargin=2pc,resume]}{\end{enumerate}}

\newenvironment{enumerate-(i)}{\begin{enumerate}[label={\upshape (\roman*)}, leftmargin=2pc]}{\end{enumerate}}

\newenvironment{enumerate-(i)-r}{\begin{enumerate}[label={\upshape (\roman*)}, leftmargin=2pc,resume]}{\end{enumerate}}

\newenvironment{enumerate-(I)}{\begin{enumerate}[label={\upshape (\Roman*)}, leftmargin=2pc]}{\end{enumerate}}

\newenvironment{enumerate-(I)-r}{\begin{enumerate}[label={\upshape (\Roman*)}, leftmargin=2pc,resume]}{\end{enumerate}}

\newenvironment{enumerate-(1)}{\begin{enumerate}[label={\upshape (\arabic*)}, leftmargin=2pc]}{\end{enumerate}}

\newenvironment{enumerate-(1)-r}{\begin{enumerate}[label={\upshape (\arabic*)}, leftmargin=2pc,resume]}{\end{enumerate}}

\begin{document}

\thanks{We would like to thank Maurice Stanley and Sy Friedman for providing us with information regarding the definability of the forcing relation in class forcing, and we would like to thank Victoria Gitman for helpful comments and discussions on some of the topics of this paper. We would also like to thank the referee for thoroughly reading through our paper, and for making many useful comments, that enabled us to very much improve this paper. The first and third author were partially supported by DFG-grant LU2020/1-1.}



\author{Peter Holy}
\address{Peter Holy, Math.\ Institut, Universit\"at Bonn,
Endenicher Allee 60, 53115 Bonn, Germany}
\email{pholy@math.uni-bonn.de}
\urladdr{}

\author{Regula Krapf}
\address{Regula Krapf, Math.\ Institut, Universit\"at Koblenz-Landau,
Universit\"atsstra\ss e 1, 56070 Koblenz, Germany}
\email{krapf@uni-koblenz.de}
\urladdr{}

\author{Philipp Schlicht}
\address{Philipp Schlicht, Math.\ Institut, Universit\"at Bonn, 
Endenicher Allee 60, 53115 Bonn, Germany}
\email{schlicht@math.uni-bonn.de}
\urladdr{}

%\title{}
\date{\today}

\title{Characterizations of Pretameness and the Ord-cc}
\begin{abstract} 
It is well known that pretameness implies the forcing theorem, and that pretameness is characterized by the preservation of the axioms of $\ZF^-$, that is $\ZF$ without the power set axiom, or equivalently, by the preservation of the axiom scheme of replacement, for class forcing over models of $\ZF$. We show that pretameness in fact has various other characterizations, for instance in terms of the forcing theorem, the preservation of the axiom scheme of separation, the forcing equivalence of partial orders and their dense suborders, and the existence of nice names for sets of ordinals. These results show that pretameness is a strong dividing line between well and badly behaved notions of class forcing, and that it is exactly the right notion to consider in applications of class forcing. Furthermore, for most properties under consideration, we also present a corresponding characterization of the $\On$-chain condition. 
\end{abstract}

\maketitle

%\setcounter{tocdepth}{1}
%\tableofcontents 

\section{Introduction}\label{sec:intro}
This paper is motivated by the question which properties of set forcing carry over to the context of class forcing, in particular for pretame notions of class forcing, and for notions of class forcing with the $\On$-chain condition. We show that the properties mentioned in the abstract hold for all pretame notions of class forcing, and that they can in fact be used to characterize pretameness and the $\On$-chain condition in various ways. A list of the main results of this paper can be found at the end of this section. In order to properly state them, we will need to introduce our setup together with some basic notation. This is mostly the same as in \cite{ClassForcing}.

\medskip

We will work with transitive second-order models of set theory, that is models of the form $\MM=\langle M,\C\rangle$, where $M$ is transitive and denotes the collection of \emph{sets} of $\MM$ and $\C$ denotes the collection of \emph{classes} of $\MM$.\footnote{Arguing in the ambient universe $\V$, we will sometimes refer to classes of such a model $\MM$ as sets, without meaning to indicate that they are sets of $\MM$. In particular this will be the case when we talk about subsets of $M$.} We require that $M\subseteq\C$ and that elements of $\C$ are subsets of $M$, and we call elements of $\C\setminus M$ \emph{proper classes} (of $\MM$).  Classical transitive first-order models of set theory are covered by our approach when we let $\C$ be the collection of classes definable over $\langle M,\in\rangle$. The theories that we will be working in will be fragments of \emph{Kelley-Morse set theory} $\KM$,
and mostly we will work within fragments of \emph{G\"odel-Bernays set theory} $\GB$. By $\GB^-$ we denote $\GB$ 
without the power set axiom (but with Collection rather than Replacement), and by $\GBC$ we denote $\GB$ together with the axiom of global choice.\footnote{For more details, see \cite{ClassForcing}. For a detailed axiomatization of $\KM$, see \cite{antos2015class}.} By a countable transitive model of some second-order theory, we mean a transitive second-order model $\MM=\langle M,\C\rangle$ of such a theory with both $M$ and $\C$ countable in $\V$. 
Most of the time, we will avoid using the power set axiom, however we will often make use of the following consequence of $\GB$:

\begin{definition}
 A model $\MM=\langle M,\C\rangle$ of $\GB^-$ \emph{has a hierarchy} if there is $C\in\C$ with $C\subseteq\On^M\times M$ such that for $\alpha\in\On^M$
 \begin{enumerate-(1)}
  \item $C_\alpha=\{x\mid\exists\beta<\alpha\;\langle\beta,x\rangle\in C\}\in M$;
  \item $M=\bigcup_{\alpha\in\On^M}C_\alpha$.
 \end{enumerate-(1)}
If $C$ is a hierarchy on $M$, then the \emph{$C$-rank} of $x\in M$, denoted $\crnk{C}{x}$, is the 
least $\alpha\in\On^M$ such that $x\in C_{\alpha+1}$. 
\end{definition}

\begin{remark}\label{rem:choice principles}
\begin{enumerate-(1)}
 \item Assume that $\MM=\langle M,\C\rangle\models\GB^-$. Suppose that $\C$ contains a well-order $\prec$ of $M$ which is \emph{set-like}, i.e.\ for every $x\in M$, $C_x=\{y\in M\mid y\prec x\}\in M$. Then the order-type of $\langle M,\prec\rangle$ is $\On^M$ and $C=\{\langle x,y\rangle\mid x\in\On^M\,\wedge\,y\prec x\}$ witnesses that $\MM$ has a hierarchy. 
 Conversely, it is easy to check that every well-order of $M$ in $\C$ that has order-type $\On^M$ is set-like.
 
 \item If $\MM\models\GB^-$ has a hierarchy and $\C$ contains a well-order of $M$, then $\C$ contains a set-like well-order of $M$. To see this, let $C\in\C$ be a hierarchy on $M$ and let $\prec\,\in\C$ be a well-order of $M$.
 Then we obtain a set-like well-order $\vartriangleleft$ of $M$ by stipulating that $x\vartriangleleft y$ if and only if $\crnk{C}{x}<\crnk{C}{y}$, or $\crnk{C}{x}=\crnk{C}{y}$ and $x\prec y$.

\item Without a hierarchy, the existence of a well-order of $M$ in $\C$ does not imply that $\C$ contains a well-order of $M$ of order-type $\On^M$. This can be seen as follows. Let $N$ be a countable transitive model of $\ZFC$ in which $\CH$ fails and such that there is a definable wellorder of $H(\omega_1)$; assuming the existence of a countable transitive model of $\ZFC$, such a model exists by \cite{MR0465866}. Let $M=H(\omega_1)^N$ and let $\C$ be the collection of subsets of $M$ that are definable over $M$. It follows that $\langle M,\C\rangle\models\GB^-$ and that $\C$ contains a well-order of $M$. Since $\CH$ fails in $N$, $\C$ however cannot contain a well-order of $M$ of order-type $\On^M$.
 
 \item If $\MM\models\GB^-$ has a hierarchy, then it satisfies \emph{representatives choice}, i.e. if for every equivalence relation 
$E\in\C$ there is $A\in\C$ and a surjective map $\pi:\dom E\ra A$ in $\C$ such that 
$\langle x,y\rangle\in E$ if and only if $\pi(x)=\pi(y)$.  
 Assuming the presence of a hierarchy, we will thus sometimes cite results from \cite{ClassForcing} that assume representatives choice, without further mention.

\end{enumerate-(1)}

\end{remark}

We define $\GBC^-$ to be the theory consisting of the axioms of $\GB^-$ together with the axiom postulating the existence of a set-like well-order. Note that by Remark \ref{rem:choice principles}, (3), this theory is strictly stronger than $\GB^-+$ global choice.

\medskip

Fix a countable transitive model $\MM=\langle M,\C\rangle$ of $\GB^-$. By a \emph{notion of class forcing} (for $\MM$) we mean a separative partial order $\PP=\langle P,\leq_\PP\rangle$ such that $P,\leq_\PP\,\in\C$.\footnote{Note that this differs from the definition of notions of class forcing in \cite{ClassForcing}, where we do not make the assumption of separativity (and in \cite{ClassForcing}, we also do not assume antisymmetry, i.e.\ we allow for the more general notion of a preorder).}
We will frequently identify $\PP$ with its domain $P$. In the following, we also fix a notion of class forcing $\PP=\langle P,\leq_\PP\rangle$ for $\MM$. Note that the assumption of separativity does not restrict us much, since if $\MM$ satisfies representatives choice, one can always pass from a notion of class forcing to its separative quotient (see \cite{ClassForcing}). 

We call $\sigma$ a \emph{$\PP$-name} if all elements of $\sigma$ are of the form $\langle\tau,p\rangle$, where 
$\tau$ is a $\PP$-name and $p\in\PP$. 
Define $M^\PP$ to be the set of all $\PP$-names that are elements of $ M$ and define $\C^\PP$ to be the set of all $\PP$-names that are elements of $\C$.
In the following, we will usually call the elements of $M^\PP$ simply \emph{$\PP$-names} and we will call the elements of $\C^\PP$ \emph{class $\PP$-names}.
If $\sigma\in M^\PP$ is a $\PP$-name, we define 
$$\rank\sigma=\sup\{\rank\tau+1\mid\exists p\in\PP\ \langle\tau,p\rangle\in\sigma\}$$ 
to be its \emph{name rank}. We will sometimes also need to use the usual set theoretic rank of some $\sigma\in M$, which we will denote as $\rnk(\sigma)$.

We say that a filter $G$ on $\PP$ is \emph{$\PP$-generic over $\MM$} if $G$ meets every dense subset of $\PP$ that is an element of $\C$. 
Given such a filter $G$ and a $\PP$-name $\sigma$, we recursively define the \emph{$G$-evaluation} of $\sigma$ as
$$\sigma^G=\{\tau^G\mid\exists p\in G\ \langle\tau,p\rangle\in\sigma\},$$
and similarly we define $\Gamma^G$ for $\Gamma\in\C^\PP$. Moreover, if $G$ is $\PP$-generic over $\MM$, then we set
$M[G]=\Set{\sigma^G}{\sigma\in M^\PP}$ and $\C[G]=\Set{\Gamma^G}{\Gamma\in\C^\PP}$. We call $\MM[G]=\langle M[G],\C[G]\rangle$ a \emph{$\PP$-generic extension} of $\MM$.

Given an $\L_\in$-formula $\varphi(v_0,\dots,v_{m-1},\vec\Gamma)$, where $\vec\Gamma\in(\C^\PP)^n$ are class name parameters, $p\in\PP$ and $\vec\sigma\in(M^\PP)^m$, we write
$p\Vdash_\PP^\MM\varphi(\vec\sigma,\vec\Gamma)$ if for every $\PP$-generic filter $G$ over $\MM$ with $p\in G$, $\langle M[G],\Gamma_0^G,\dots,\Gamma_{n-1}^G\rangle\models\varphi(\sigma_0^G,\dots,\sigma_{m-1}^G,\Gamma_0^G,\dots,\Gamma_{n-1}^G)$. 

A fundamental result in the context of set forcing is the \emph{forcing theorem}. It consists of two parts, the first one of which, the so-called \emph{definability lemma}, states that the forcing relation is definable in the ground model, and the second part, denoted as the \emph{truth lemma}, says that every formula which is true in a generic extension $M[G]$ is forced by some condition in the generic filter $G$. In the context of second-order models of set theory, this has the following natural generalization:

\begin{definition}\label{def:ft}
 Let $\varphi\equiv\varphi(v_0,\ldots,v_{m-1},\vec\Gamma)$ be an $\L_\in$-formula with class name parameters $\vec\Gamma\in(\C^\PP)^n$. 
 \begin{enumerate}
  \item We say that \emph{$\PP$ satisfies the definability lemma for $\varphi$ over $\MM$} if 
   \begin{equation*}
  \Set{\langle p,\sigma_0,\ldots,\sigma_{m-1}\rangle\in P\times(M^\PP)^m}{p\Vdash^\MM_\PP\varphi(\sigma_0,\ldots,\sigma_{m-1},\vec\Gamma)}\in\C. 
 \end{equation*} 
 \item We say that \emph{$\PP$ satisfies the truth lemma for $\varphi$ over $\MM$} if for all $\sigma_0,\ldots,\sigma_{m-1}\in M^\PP$, and every filter $G$ which is $\PP$-generic over $\MM$ with 
  \begin{equation*}
    \MM[G]\models\varphi(\sigma_0^G,\ldots,\sigma_{m-1}^G,\Gamma_0^G,\dots,\Gamma_{n-1}^G), 
  \end{equation*}
  there is $p\in G$ with $p\Vdash^{\MM}_\PP\varphi(\sigma_0,\ldots,\sigma_{m-1},\vec\Gamma)$.
  \item We say that \emph{$\PP$ satisfies the forcing theorem for $\varphi$ over $\MM$} if $\PP$ satisfies both the definability lemma and the truth lemma for $\varphi$ over $\MM$. We say that \emph{$\PP$ satisfies the forcing theorem} if $\PP$ satisfies the forcing theorem for all first order formulas over $\MM$.
 \end{enumerate}
 \end{definition}
 
We will frequently make use of the following observation.

\begin{observation}\label{defsubsetclasses}
  If $\PP$ is a notion of class forcing for $\MM$ that satisfies the forcing theorem, then 
  every class which is first-order definable over $\MM[G]$ is an element of $\C[G]$.
\end{observation}
\begin{proof}
  Let $\varphi(v,x,\vec C)$ be a first-order formula with parameter $x\in M[G]$ and class parameters $\vec C\in(\C[G])^n$. Let $\sigma$ in $M^\PP$ be a name for $x$, let $\vec\Gamma\in(\C^\PP)^n$ be class names for the elements of $\vec C$.
Let $\Gamma=\{\langle\mu,p\rangle\mid p\Vdash_\PP\varphi(\mu,\sigma,\vec\Gamma)\}\in\C$. The claim follows, as in $M[G]$, $\Gamma^G=\{y\mid\varphi(y,x,\vec C)\}$.
\end{proof}

We identify sequences of the form $\langle C_i\mid i\in I\rangle$ for classes $\C_i\in\C$ and $I\in\C$ with their \emph{code} $\{\langle c,i\rangle\mid c\in C_i\,\land\,i\in I\}$. In particular, we say that such a sequence is an element of $\C$ if its code is.

\medskip

Note that by \cite[Theorem 1.3]{ClassForcing}, the forcing theorem can fail in class forcing; in fact, even the truth lemma may fail for atomic formulae (see \cite[Theorem {1.5}]{ClassForcing}). There are, however, many known properties of class forcing notions which guarantee that the forcing theorem holds (see for example \cite{MR0345819}, \cite{stanley_thesis}, \cite{MR1780138} and \cite{ClassForcing}). The most prominent such property is that of \emph{pretameness}. The term \emph{pretameness} was first used by Sy Friedman (\cite{MR1780138}), however the concept already appears in unpublished work of Maurice Stanley (\cite{stanley_classforcing}), under the name of \emph{predensity reduction}, and is implicit in the work of Andrzej Zarach (\cite{MR0345819}), in his concept of \emph{C-3-genericity}, and the similar \emph{weak antichain property} appears in \cite{stanley_thesis}.

\begin{definition}
A notion of (class) forcing $\PP$ for $\MM$ is \emph{pretame} for $\MM\models\GB^-$ if for every $p\in\PP$ and for every sequence of dense subclasses $\langle D_i\mid i\in I\rangle\in\C$ of $\PP$ with $I\in M$, there is $q\leq_\PP p$ and $\langle d_i\mid i\in I\rangle\in M$ such that for every $i\in I$, $d_i\subseteq D_i$ and $d_i$ is predense below $q$.
\end{definition}

It was first shown by Stanley in his PhD thesis (\cite{stanley_thesis}) that pretame notions of class forcing satisfy the forcing theorem over any model of $\ZF^-$. A much more well-known account, however over the stronger base theory $\ZF$, was published by Friedman as \cite[Theorem 2.18]{MR1780138}.\footnote{It is in fact easy to observe that $\ZF$ can be replaced by $\ZF^-$ together with the existence of a hierarchy for the arguments of that proof to go through.} For these reasons, and for the benefit of the reader, we will provide a proof of this result in a generalized setting at the beginning of Section \ref{sec:ft}, which we independently discovered before learning about Stanley's results, and which turned out to essentially follow the line of argument in a version of his proof in the unpublished \cite{stanley_classforcing}. Furthermore, we will show that in a certain sense, pretameness can in fact be characterized by the forcing theorem.

If $\MM\models X\subseteq\KM$, we say that a notion of class forcing $\PP$ for $\MM$ \emph{preserves} $X$ if $\MM[G]\models X$ for every $\PP$-generic filter $G$ over $\MM$. Stanley observed (\cite{stanley_thesis},\cite{stanley_classforcing}) that pretameness characterizes the preservation of $\ZF^-$ over models of $\ZF$. We will provide an easy generalization of these arguments to the context of second-order models of set theory in Section \ref{sec:axioms}.

\begin{definition}
We say that a notion of class forcing $\PP$ satisfies the \emph{$\On$-chain condition} (or simply \emph{$\On$-cc}) over $\MM$
if every antichain of $\PP$ which is in $\C$ is already in $M$.  
\end{definition}

Note that if $\MM\models\GBC^-$, then for notions of class forcing for $\MM$, the $\On$-cc is strictly stronger than pretameness. On the other hand, over any model of $\GB^-+\mathsf{AC}$ with a hierarchy that does not satisfy global choice, there is a non-pretame notion of class forcing which satisfies the $\On$-cc (see \cite[Lemma 4.1.7]{thesis}).

\medskip

A property that is closely related to the forcing theorem, and that can be used to characterize pretameness and the $\On$-cc, is the existence of a Boolean completion. We distinguish between two types of Boolean completions, depending on whether suprema exist for all sets or for all classes of conditions. 

\begin{definition}
If $\BB$ is a Boolean algebra, then $\BB$ is
\begin{enumerate-(1)}
  \item \emph{$M$-complete} if the 
supremum $\sup_\BB A$ exists in $\BB$ for every $A\in M$ with $A\subseteq\BB$. 
\item \emph{$\C$-complete} if the supremum $\sup_\BB A$ exists in $\BB$ for every $A\in\C$ with $A\subseteq\BB$. 
\end{enumerate-(1)}
\end{definition}

\begin{definition}\label{def:bc}
We say that \emph{$\PP$ has a Boolean $M$-completion in $\MM$} if there is an $M$-complete Boolean 
$\BB=\langle B,0_\BB,1_\BB\,\neg,\wedge,\vee\rangle$ such that $B$, all Boolean operations of $\BB$ and an injective dense embedding from $\PP$ into $\BB\setminus\{0_\BB\}$ are elements of $\C$. Similarly, we define a \emph{Boolean $\C$-completion} of $\PP$ to be a Boolean $M$-completion $\BB$ of $\PP$ which additionally is $\C$-complete.  
\end{definition}

In set forcing, Boolean completions are unique: If $\BB_0$ and $\BB_1$ are 
both Boolean completions of $\PP$ and $e_0:\PP\ra\BB_0$ and $e_1:\PP\ra\BB_1$ are dense embeddings, then 
one can define an isomorphism from $\BB_0$ to $\BB_1$ by setting $f(b)=\sup\{e_1(p)\mid p\in\PP\wedge e_0(p)\leq b\}$ for $b\in\BB_0$. Moreover,
$f$ fixes $\PP$ in the sense that $f(e_0(p))=e_1(p)$ for every condition $p\in\PP$. 
In class forcing, this proof works only for Boolean $\C$-completions (provided that they exist). It follows from results in \cite[Section 9]{ClassForcing} that Boolean $M$-completions need not be \emph{unique} in the following sense.

\begin{definition}\label{def:unique bc}
 We say that a notion of class forcing \emph{$\PP$ has a unique Boolean $M$-completion in $\MM$}, if $\PP$ has a Boolean $M$-completion 
 $\BB_0$ in $\MM$ and for every other Boolean $M$-completion $\BB_1$ of $\PP$ in $\MM$ 
 there is an isomorphism in $\V$ between $\BB_0$ and $\BB_1$ which fixes $\PP$. The property that \emph{$\PP$ has a unique Boolean $\C$-completion} is defined correspondingly. 
\end{definition}

In Section \ref{sec:bc}, we will investigate the relationship between the existence of Boolean completions on the one hand, and pretameness and the $\On$-cc on the other. 

\medskip

A technique that we will use in many places throughout this paper is adding suprema of proper subclasses of conditions to class forcing notions $\PP$. 
We will make use of the following characterization of suprema: For a class $A\in\C$ with $A\subseteq\PP$, we say that $p\in\PP$ is the \emph{supremum} of $A$ (denoted $p=\sup_\PP A$) if and only if 
\begin{enumerate-(1)}
 \item $\forall a\in A\ a\leq_\PP p$ and
 \item $A$ is predense below $p$. 
\end{enumerate-(1)}
It is easy to check that this corresponds to the usual definition of suprema as least upper bounds.

\medskip 

We now describe a general method of how to extend a notion of class forcing $\PP$ by adding suprema. 
Let $S=\langle X_i\mid i\in I\rangle\in\C$ with $I\in M$ be a sequence of subclasses of $\PP$.
Making use of a suitable bijection in $\C$, we may assume that $P\cap I=\emptyset$. Now let $\PP_S$ be the forcing notion with domain $P_S=P\cup I$ ordered by 
\begin{align*}
 i\leq_{\PP_S} p&\Llr\forall q\in X_i\ q\leq_\PP p,\\
 p\leq_{\PP_S} i&\Llr X_i\text{ is predense below $p$ in }\PP,\\
 i\leq_{\PP_S} j&\Llr\forall q\in X_i\ q\leq_{\PP_S}j.
\end{align*}
For $i\in I$, we will usually write $\sup X_i$ rather than $i$. 
In case that $\sup X_i$ already exists in $\PP$, or that $\sup X_i\leq_{\PP_S}\sup X_j$ and $\sup X_j\leq_{\PP_S}\sup X_i$ for some $i\neq j$, instead of $\PP_S$ we need to consider the quotient of $\PP_S$ by the equivalence relation
$p\sim q$ iff $p\leq_{\PP_S}q$ and $q\leq_{\PP_S}p$ for $p,q\in P\cup I$, in order to obtain a separative partial order. Since $I\in M$ and $\PP$ is separative, all equivalence classes are set-sized, hence this can easily be done, and we will identify $\PP_S$ with this quotient in this case. 
We call $\PP_S$ 
the \emph{forcing notion obtained from $\PP$ by adding $\sup X_i$ for all $i\in I$}. Note that by construction, $\PP$ is dense in $\PP_S$. 

\begin{lemma}\label{lemma:add suprema ft}
Suppose that $\PP$ is a notion of class forcing which satisfies the forcing theorem. 
 If $S\in\C$ is a finite sequence of subclasses of $\PP$, then $\PP_S$ satisfies the forcing theorem. 
\end{lemma}

\begin{proof}
 This is an easy generalization of \cite[Lemma 9.3]{ClassForcing}.
\end{proof}

It will follow from the proof of Theorem \ref{thm:failure ft} that this may fail for infinite sequences of subclasses of $\PP$.

\begin{example}
 We will frequently use the following notion of class forcing to motivate our results. Given $\MM=\langle M,\C\rangle\models\GB^-$, let $\Col(\omega,\On)^M$ denote the notion of class forcing with conditions of the form $p\colon \dom{p}\to\On^M$ for $\dom p\subseteq\omega$ finite, ordered by reverse inclusion. It follows from  \cite[Lemma 2.2, Lemma 6.3 and Theorem 6.4]{ClassForcing} that $\Col(\omega,\On)^M$ satisfies the forcing theorem. 
However, $\Col(\omega,\On)^M$ is non-pretame as witnessed by the sequence $\langle D_n\mid n\in\omega\rangle$, where $D_n$ is the dense class of all conditions $p$ such that $n\in\dom{p}$. Another way to see this is to observe that any $\Col(\omega,\On)^M$-generic filter gives rise to a cofinal sequence from $\omega$ to $\On^M$, and thus induces a failure of Replacement in the generic extension (see Section \ref{sec:axioms}). 
\end{example}

\subsection*{Results}
\begin{notation}
 Let $\MM\models\GB^-$ and let $\Psi$ be some property of a notion of class forcing $\PP$ for $\MM=\langle M,\C\rangle$. We say that $\PP$ \emph{densely} satisfies $\Psi$ if every notion of class forcing $\QQ$ for $\MM$, for which there
 is a dense embedding in $\C$ from $\PP$ into $\QQ$, satisfies the property $\Psi$.
\end{notation}

For the sake of simplicity, if there is a dense embedding in $\C$ from $\PP$ into $\QQ$, 
we will frequently assume that $\PP$ is a suborder of $\QQ$. This does not constitute a restriction, since in the above situation, $\PP$ is always isomorphic to a (dense) suborder of $\QQ$.

The following two theorems summarize the results of the present paper. Regarding Theorem \ref{pretamenessth} below, in a slightly less general context, that pretameness implies both (1) and (2) and that (1) implies pretameness is due to Maurice Stanley (\cite{stanley_thesis},\cite{stanley_classforcing}), and the equivalence of (4) and (5) is an immediate consequence of \cite[Theorem 5.5]{ClassForcing}.

\begin{theorem}\label{pretamenessth}Let $\MM=\langle M,\C\rangle$ be a countable transitive model of $\GBC^-$, and let $\PP$ be a notion of class forcing for $\MM$.
The following properties (over $\MM$) are equivalent to the pretameness of $\PP$ over $\MM$, where we additionally require the non-existence of a first-order truth predicate for (4) and (5), and for (7) we assume that $\MM\models\KM$, which is notably incompatible to the assumptions used for (4) and (5). 
\begin{enumerate-(1)} 
 \item $\PP$ preserves $\GB^-$/Collection/Replacement.
 \item $\PP$ satisfies the forcing theorem and preserves Separation.
 \item $\PP$ satisfies the forcing theorem and does not add a cofinal/surjective function from some $\gamma\in\On^M$ to $\On^M$. 
  \item $\PP$ densely satisfies the forcing theorem. 
 \item $\PP$ densely has a Boolean $M$-completion.
 \item $\PP$ satisfies the forcing theorem and produces the same generic extensions as $\QQ$ for every forcing notion $\QQ$ such that $\C$ contains a dense embedding from $\PP$ into $\QQ$.\footnote{More precisely, if $\pi:\PP\ra\QQ$ in $\C$ is a dense embedding and $H$ is $\QQ$-generic over $\MM$, then $M[\pi^{-1}[H]]=M[H]$.}
 \item $\PP$ densely has the property that every set of ordinals in any of its generic extensions has a nice name.
\end{enumerate-(1)}
\end{theorem}

Given notions of class forcing $\PP$ and $\QQ$ for $\MM=\langle M,\C\rangle\models\GB^-$, given $\pi\in\C$ such that $\pi$ is a dense embedding from $\PP$ to $\QQ$, and given a $\PP$-name $\sigma$, we recursively define $\pi(\sigma)=\{\langle\pi(\tau),\pi(p)\rangle\mid\langle\tau,p\rangle\in\sigma\}$. 

\begin{theorem}\label{ordccth}
Let $\MM=\langle M,\C\rangle$ be a countable transitive model of $\GB^-$ such that $\MM\models\GBC^-$, and let $\PP$ be a notion of class forcing that satisfies the forcing theorem. The following properties (over $\MM$) are equivalent:
\begin{enumerate-(1)}
 \item $\PP$ satisfies the $\On$-cc. 
 \item $\PP$ satisfies the maximality principle.\footnote{See Definition \ref{def:mp}. }
 \item $\PP$ densely has a unique Boolean $M$-completion.
 \item $\PP$ has a Boolean $\C$-completion.
 \item If there are $\QQ,\pi\in\C$ such that $\pi$ is a dense embedding from $\PP$ to $\QQ$ and $\sigma\in M^\QQ$, then there is $\tau\in M^\PP$ with $\one_\QQ\Vdash_\QQ\sigma=\pi(\tau)$. 
 \item $\PP$ densely has the property that whenever $\one_\PP\Vdash_\PP\sigma\subseteq\check\alpha$ for some $\sigma\in M^\PP$ and $\alpha\in\On^M$, then there is a nice $\PP$-name $\tau$ such that $\one_\PP\Vdash_\PP\sigma=\tau$.  
\end{enumerate-(1)}
\end{theorem}

Whether pretameness implies Statement (6) of Theorem \ref{pretamenessth} was a question (that turned out to have an easy positive answer) posed to us by Victoria Gitman at the European Set Theory Workshop in Cambridge in the summer of 2015, and was one of the starting points of the research that we present in this paper.

\section{The forcing theorem}\label{sec:ft}
In this section, we characterize pretameness in terms of the forcing theorem. 
Our first goal will be to verify that pretame notions of class forcing satisfy the forcing theorem. Let $\MM=\langle M,\C\rangle$ be a fixed countable transitive model of $\GB^-$.

\begin{definition}Let $\PP$ be a notion of class forcing for $\MM$ and let $p,q\in\PP$. 
\begin{enumerate-(1)}
 \item If $p$ and $q$ are compatible in $\PP$, we say that \emph{$D\subseteq\PP$ is dense below $p\land q$} if $\forall r\leq_\PP p,q\,\exists s\leq_\PP r\,( s\in D)$.\footnote{We do not assume that $\PP$ is closed under meets; $p\land q$ is merely a symbol here and in the following, and is not supposed to denote the actual meet of $p$ and $q$ in $\PP$, which may not exist.}
 \item We say that \emph{$d\subseteq\PP$ is predense below $p\land q$} if $\forall r\leq_\PP p,q\,\exists s\in d\,( r\parallel_\PP s)$. 
\end{enumerate-(1)}
  Moreover, if $p\perp_\PP q$, we make the additional convention that only the empty set is predense below $p\land q$, while any set is dense below $p\land q$.\footnote{While this may at first seem like a somewhat weird convention, it will save us a lot of notational cumbersomeness in the arguments to follow.}
\end{definition}

We will make use of the following equivalent formulation of pretameness:

\begin{remark}\label{rem:pretamewedge}
A notion of class forcing $\PP$ is pretame for $\MM$ if and only if for every $p\in\PP$ and for all sequences
$\langle D_i\mid i\in I\rangle\in\C$ and $\langle s_i\mid i\in I\rangle\in M$ such that each $D_i$ is dense below $p\wedge s_i$, there are $q\leq_\PP p$ and $\langle d_i\mid i\in I\rangle\in M$ such that each $d_i\subseteq D_i$ is predense below $q\wedge s_i$. 
\end{remark}

Let us fix some notion of class forcing $\PP$ for $\MM$. By \cite[Theorem 4.3]{ClassForcing}, in order to verify the forcing theorem for $\PP$, it suffices to check that the forcing relation for some atomic $\mathcal L_\in$-formula is definable over $\MM$. Let $\sigma$ and $\tau$ be $\PP$-names in $\MM$, and let $p\in\PP$ be a condition. In set forcing, one may define the syntactic forcing relation $\forces^*_\PP$ for atomic formulae recursively, by stipulating the following (essentially as in \cite[Chapter VII, Definition 3.3]{MR597342}).

\medskip

\begin{enumerate}
    \item[(A)] $p\forces^*_\PP\sigma\subseteq\tau$ if and only if the following property (a) holds: for all $\langle\rho,s\rangle\in\sigma$,
   \[\{q\mid q\forces^*_\PP\rho\in\tau\}\textrm{ is dense below }p\land s\textrm{ in }\PP.\]
   \item [(=)] $p\forces^*_\PP\sigma=\tau$ if and only if $p\forces^*_\PP\sigma\subseteq\tau$ and $p\forces^*_\PP\tau\subseteq\sigma$.
  \item[(B)] $p\forces^*_\PP\sigma\in\tau$ if and only if the following property (b) holds:
   \[\{q\mid \exists\langle\rho,s\rangle\in\tau\ [q\leq_\PP s\,\land\,q\forces^*_\PP\sigma=\rho]\}\textrm{ is dense below }p\textrm{ in }\PP.\]
\end{enumerate}

To verify the definability of the semantic forcing relation $\forces_\PP\,=\,\forces_\PP^\MM$ for atomic formulae in the classical case of set forcing, one then proceeds to show that the relation $\forces^*_\PP$, which in the case of set forcing is easily seen to be a definable relation, concides with $\forces_\PP$. Such string of arguments may be found in \cite[Chapter VII, \S 3]{MR597342}, leading up to its \cite[Chapter VII, Theorem 3.6]{MR597342}. Provided that we have a class relation $\forces^*_\PP$ that satisfies Properties (A), (=) and (B) w.r.t.\ a notion of class forcing $\PP$, it is easy to observe that exactly the same arguments yield that $\forces^*_\PP\,=\,\forces_\PP$ for atomic formulae. That is, classical arguments yield the following.

\begin{lemma}
  Assume that $\PP$ is a notion of class forcing for a countable transitive model $\MM=\langle M,\C\rangle\models\GB^-$, and assume that $\forces^*_\PP\;\subseteq\PP\times M^\PP\times\{\subseteq,\in\}\times M^\PP$ is a relation in $\C$. If $\forces^*_\PP$ satisfies Properties (A), (=) and (B), then $\forces^*_\PP$ coincides with the semantic forcing relation $\forces_\PP$ for atomic formulae in the forcing language of $\PP$ over $\MM$.\hfill\qedsymbol
\end{lemma}

We are now ready to provide a proof of the following theorem.

\begin{theorem}[Stanley]\label{thm:pretame ft}
  Let $\MM$ be a countable transitive model of $\GB^-$, and let $\PP$ be a notion of class forcing for $\MM$. If $\PP$ is pretame over $\MM$, then $\PP$ satisfies the forcing theorem over $\MM$.
\end{theorem}

\begin{proof}
We will show that we can define a class $\forces^*_\PP$ within $\MM$ that satisfies properties (A), (=) and (B). We will not bother with defining $\forces^*_\PP$ w.r.t.\ equality, considering its instances as mere abbreviations for two statements concerning the $\subseteq$-relation.

However, we will need to first define yet another auxiliary forcing relation, that collects forcing statements for which we have (set-sized) witnesses in $M$; we may call this the \emph{witnessed forcing} relation (for atomic formulae) of $\MM$. For $R\in\{\subseteq,\in\}$, we define that $p\forces^{**}_\PP\sigma R\tau$ if there is a set $Y\subseteq\PP\times M^\PP\times\{\subseteq,\in\}\times M^\PP$ in $M$, and a set $\vec D=\langle d_y\mid y\in Y\rangle$ in $M$ with the following properties:
\begin{enumerate}
  \item If $y\in Y$ is of the form $y=\langle q,\mu,\in,\nu\rangle$, then $d_y$ is a predense subset of $\PP$ below $q$.
  \item $\forall\langle q,\mu,\in,\nu\rangle\in Y\,\forall r\in d_{\langle q,\mu,\in,\nu\rangle}\,\exists\langle\rho,s\rangle\in\nu\ \left[r\leq_\PP s\,\land\,\langle r,\mu,\subseteq,\rho\rangle\in Y\land\langle r,\rho,\subseteq,\mu\rangle\in Y\right]$.
  \item If $y\in Y$ is of the form $y=\langle q,\mu,\subseteq,\nu\rangle$, then $d_y$ is of the form $d_y=\langle d_y^{\rho,s}\mid\langle\rho,s\rangle\in\mu\rangle$ and each $d_y^{\rho,s}$ is a predense subset of $\PP$ below $q\land s$ whenever $q\parallel_\PP s$, and the empty set otherwise.
  \item $\forall\langle q,\mu,\subseteq,\nu\rangle\in Y\,\forall\langle\rho,s\rangle\in\mu\,\forall r\in d_{\langle q,\mu,\subseteq,\nu\rangle}^{\rho,s}\ \langle r,\rho,\in,\nu\rangle\in Y$.
  \item $\langle p,\sigma,R,\tau\rangle\in Y$.
\end{enumerate}
In the above, we say that $Y$ and $\vec D$ \emph{witness} that $p\forces^{**}_\PP\!\sigma R\tau$. We write $p\forces^{**}_\PP\sigma=\tau$ to denote the statement that $p\forces^{**}_\PP\sigma\subseteq\tau$ and $p\forces^{**}_\PP\tau\subseteq\sigma$. For $R\in\{\subseteq,\in,=\}$,  we define that $p\forces^*_\PP\sigma R\tau$ if the class $\{q\leq_\PP p\mid q\Vdash^{**}_\PP\sigma R \tau\}$ is dense below $p$. 
We will verify that $\forces^*_\PP$ satisfies properties (A), (=) and (B), thus finishing the proof of the theorem.
We start with three easy claims.

\begin{definition}
  Given a $\PP$-name $\sigma$, we define the names appearing in (the transitive closure of) $\sigma$ by induction on name rank as
\[\tn(\sigma)=\bigcup\{\{\tau\}\cup\tn(\tau)\mid\exists p\ \langle\tau,p\rangle\in\sigma\}.\]
\end{definition}

\begin{claim}\label{easy2}
Let $p\in\PP$ and $\sigma,\tau\in M^\PP$. Then the following statements hold.
\begin{enumerate}
 \item[(i)] If $\langle\sigma,p\rangle\in\tau$ then $p\Vdash^{**}_\PP\sigma\in\tau$.
 \item[(ii)] If $Y,\vec D$ witness that $p\forces^{**}_\PP\sigma R\tau$, then for every element $\langle q,\mu,R',\nu\rangle\in Y$, $q\forces^{**}_\PP\mu R'\nu$.
\end{enumerate}
\end{claim}

\begin{proof}
 For (i), note that 
 $$Y=\{\langle q,\mu,\in,\nu\rangle\mid\langle\mu,q\rangle\in\nu\wedge\nu\in\{\tau\}\cup\tn(\tau)\}\cup\{\langle q,\mu,\subseteq,\mu\rangle\mid\langle\mu,q\rangle\in\tn(\tau)\}$$
 and $\vec D=\langle d_y\mid y\in Y\rangle$ with 
 $d_{\langle q,\mu,\in,\nu\rangle}=\{q\}$ and $d_{\langle q,\mu,\subseteq,\mu\rangle}^{\rho,s}=\{s\}$ witness that $p\forces^{**}_\PP\sigma\in\tau$.
 
 Property (ii) is immediate.
\end{proof}

\begin{claim}\label{easy}
The following statements hold for any $p\in\PP$, set names $\sigma,\tau\in M^\PP$ and $R\in\{\subseteq,\in,=\}$.
  \begin{enumerate}
    \item[(i)] If $p\forces^{**}_\PP\sigma R\tau$ and $q\leq_\PP p$ then $q\forces^{**}_\PP\sigma R\tau$.
    \item[(ii)] If $p\forces^{**}_\PP\sigma R\tau$ then $p\forces^*_\PP\sigma R\tau$.
    \item[(iii)] If $p\forces^*_\PP\sigma R\tau$ and $q\leq_\PP p$ then $q\forces^*_\PP\sigma R\tau$.
  \end{enumerate}
\end{claim}
\begin{proof}
   In order to verify Property (i), pick $Y$ and $\vec D=\langle d_y\mid y\in Y\rangle$ witnessing that $p\forces^{**}_\PP\sigma R\tau$, and let $q\leq_\PP p$. If it happens to be the case that $\langle q,\sigma,R,\tau\rangle\in Y$, then we are done by Claim \ref{easy2}, (ii). Otherwise, let $Z=Y\cup\{\langle q,\sigma,R,\tau\rangle\}$, and let $\vec E=\vec D\cup\{\langle\langle q,\sigma,R,\tau\rangle,d_{\langle p,\sigma,R,\tau\rangle}\rangle\}$. Then $Z$ and $\vec E$ witness that $q\forces^{**}_\PP\sigma R\tau$.
   
   (ii) is immediate from (i), and (iii) follows directly from the definition of $\forces^{**}_\PP$.
\end{proof}

\begin{claim}\label{witnessunionpretame}
  Assume that we are given a sequence $\langle\langle p_i,\sigma_i,R_i,\tau_i,Y_i,\vec D_i\rangle\mid i\in I\rangle\in M$ such that for every $i\in I$, $Y_i$ and $\vec D_i=\langle (d^i)_y\mid y\in Y_i\rangle$ witness that $p_i\forces^{**}_\PP\sigma_i R_i\tau_i$. Then there are $Y$ and $\vec D$, both in $M$, witnessing that for every $i\in I$, $p_i\forces^{**}_\PP\sigma_i R_i\tau_i$.
\end{claim}
\begin{proof}
  Let $Y=\bigcup_{i\in I}Y_i$ and $\vec D=\langle d_y\mid y\in Y\rangle$,
  where 
  $$d_y=\begin{cases}
   \bigcup\{(d^i)_y\mid i\in I, y\in Y_i\},&\textrm{if }y=\langle q,\mu,\in,\nu\rangle,\\
   \langle\bigcup\{(d^i)_y^{\rho,s}\mid i\in I,y\in Y_i\}\mid\langle\rho,s\rangle\in\mu\rangle,&\textrm{if }y=\langle q,\mu,\subseteq,\nu\rangle.
  \end{cases}$$
  It is straightforward to verify properties (1)--(5) above in order to show that $Y$ and $\vec D$ witness that  for every $i\in I$, $p_i\forces^{**}_\PP\sigma_i R_i\tau_i$.
\end{proof}

Now we are ready to prove that the relation $\Vdash^*_\PP$ satisfies properties (A), (=) and (B). Let $p\in\PP$ be a condition and let $\sigma,\tau\in M^\PP$ set names.

Assume that $p\forces^*_\PP\sigma\subseteq\tau$ and let $\langle\rho,s\rangle\in\sigma$ such that $p\parallel_\PP s$. Take $q\in\PP$ with $q\leq_\PP p,s$. By our assumption, there is $r\leq_\PP q$ with $r\Vdash_\PP^{**}\sigma\subseteq\tau$. Let $Y$ and $\vec D$ witness this. Then $y=\langle r,\sigma,\subseteq,\tau\rangle\in Y$ and hence $d_y^{\rho,s}$ is predense below $r$. Let $r'\parallel_\PP r$ with $r'\in d_y^{\rho,s}$. By Property (4), $\langle r',\rho,\in,\tau\rangle\in Y$. In particular, Claim \ref{easy2} (ii) implies that $r'\Vdash_\PP^{**}\rho\in\tau$. By Claim \ref{easy} (i), $t\Vdash_\PP^{**}\rho\in\tau$ for every condition $t$ with $t\leq_\PP r,r'$, proving (a).

Conversely, suppose that Property (a) holds for $p,\sigma$ and $\tau$ and let $q\leq_\PP p$. We need to find $r\leq_\PP q$ with $r\Vdash_\PP^{**}\sigma\subseteq\tau$. By Property (a), for each $\langle\rho,s\rangle\in\sigma$ the class
$$D^{\rho,s}=\{r\leq_\PP q\mid r\Vdash_\PP^{**}\rho\in\tau\}\in\C$$
is dense below $q\wedge s$. Applying pretameness in the sense of Remark \ref{rem:pretamewedge}, we can find $r\leq_\PP q$ and a sequence $\langle d^{\rho,s}\mid\langle\rho,s\rangle\in\sigma\rangle\in M$ such that each $d^{\rho,s}\subseteq D^{\rho,s}$ is predense below $r\wedge s$.  
Using Collection and Separation, we can find a set $X\in M$ such that each element of $X$ is of the form $\langle Y,\vec D\rangle$, where $Y,\vec D$ witness that $t\Vdash_\PP^{**}\rho\in\tau$ for some $\langle\rho,s\rangle\in\sigma$ and $t\in d^{\rho,s}$. Using Claim \ref{witnessunionpretame}, we can amalgamate the elements of $X$ to obtain $Y,\vec D$ which witness simultaneously that $t\Vdash_\PP^{**}\rho\in\tau$ for every such $t\in d^{\rho,s}$ with $\langle\rho,s\rangle\in\sigma$. 
Now let $Z=Y\cup\{\langle r,\sigma,\subseteq,\tau\rangle\}$, let $d_{\langle r,\sigma,\subseteq,\tau\rangle}=\langle d^{\rho,s}\mid\langle\rho,s\rangle\in\sigma\rangle$, and let $\vec E=\vec D\cup\{\langle\langle r,\sigma,\subseteq,\tau\rangle,d_{\langle r,\sigma,\subseteq,\tau\rangle}\rangle\}$. By the definition of $\forces_\PP^{**}$, $Z$ and $\vec E$ witness that $r\forces_\PP^{**}\sigma\subseteq\tau$. But this means that $p\forces_\PP^*\sigma\subseteq\tau$, as desired.

Assume that $p\forces^*_\PP\sigma\in\tau$ and let $q\leq_\PP p$. By our assumption, there is $r\leq_\PP q$ such that $r\forces^{**}_\PP\sigma\in\tau$. Let $Y$ and $\vec D$ witness this. Then $\langle r,\sigma,\in,\tau\rangle\in Y$ and hence $d_{\langle r,\sigma,\in,\tau\rangle}$ is predense below $r$. Let $r'\in d_{\langle r,\sigma,\in,\tau\rangle}$ be such that $r'$ and $r$ are compatible. By Property (2), there is $\langle\rho,s\rangle\in\tau$ such that $r'\leq_\PP s$ and $\langle r',\sigma,\subseteq,\rho\rangle,\langle r',\rho,\subseteq,\sigma\rangle\in Y$. This means that $r'\forces^{**}_\PP\sigma=\rho$. Let $t\leq_\PP r,r'$ witness that $r$ and $r'$ are compatible. Then by Claim \ref{easy}, (i), $t\forces^{**}_\PP\sigma=\rho$, and hence by Claim \ref{easy}, (ii), $t\forces^{*}_\PP\sigma=\rho$. This shows that Property (b) holds.

Conversely, suppose that Property (b) holds. Then by the definition of $\forces_\PP^*$,   $$D=\{r\leq_\PP p\mid\exists\langle\rho,s\rangle\in\tau\ \left[r\leq_\PP s\,\land\,r\forces^{**}_\PP\sigma=\rho\right]\}\in\C$$ is a dense class of conditions below $p$. We claim that $r\forces^{**}_\PP\sigma\in\tau$ for every $r\in D$. Pick some condition $r\in D$, and let $\langle\rho,s\rangle\in\tau$ witness that $r\in D$. By Claim \ref{easy2} (i), $s\forces^{**}_\PP\rho\in\tau$ and hence by Claim \ref{easy} (i), $r\forces^{**}_\PP\rho\in\tau$. Using Claim \ref{witnessunionpretame}, we can take $Y$ and $\vec D$ witnessing that $r\forces^{**}_\PP\rho\in\tau$ and that $r\forces_\PP^{**}\sigma=\rho$. 
Let $Z=Y\cup\{r,\sigma,\in,\tau\}$ and $\vec E=\vec D\cup\{\langle\langle r,\sigma,\in,\tau\rangle,\{r\}\rangle\}$. Then $Z$ and $\vec E$ witness that $r\forces^{**}_\PP\sigma\in\tau$. But this means exactly that $p\forces^*_\PP\sigma\in\tau$, as desired.
\end{proof}

In \cite[Theorem 1.3]{ClassForcing}, it is shown that the forcing theorem may fail in class forcing. The notion of forcing used to prove this adds a binary predicate $E$ on $\omega$ so that $\langle\omega,E\rangle$ isomorphic to $\langle M,\in\rangle$. The idea is that if the forcing theorem was satisfied one would obtain a first-order truth predicate in the ground model. 
In this section, we will prove a density version of this theorem. 
The following easy observation will be a key ingredient for our proof.

\begin{lemma}\label{lemma:pretame cofinal function}
 Suppose that $\MM$ is a countable transitive model of $\GB^-$ with a hierarchy, and let $\PP$ be a notion of class forcing for $\MM$ which satisfies the forcing theorem. Then $\PP$ is pretame if and only if there exist no set $a\in M$, $\dot F\in\C^\PP$ and $p\in\PP$ such that $p\Vdash_\PP\anf{\dot F:\check a\ra\On^M\text{ is cofinal}}$.
\end{lemma}

\begin{proof}
Suppose first that $\PP$ is pretame, that $a\in M$, and that $p\Vdash_\PP\anf{\dot F:\check a\ra\On^M}$. Let
$$D_x=\{q\leq_\PP p\mid\exists\gamma\in\On^M\,q\Vdash_\PP\dot F(\check x)=\check\gamma\}$$
for $x\in a$. By pretameness, there are $q\leq_\PP p$ and a sequence $\langle d_x\mid x\in a\rangle\in M$ such that each $d_x\subseteq D_x$ is predense below $q$. Let $$\beta=\sup\{\gamma+1\mid\gamma\in\On^M\,\land\,\exists x\in a\,\exists r\in d_x\ r\Vdash_\PP\dot F(\check x)=\check\gamma\}\in\On^M.$$
Then $q\Vdash_\PP\ran{\dot F}\subseteq\check\beta$, hence $p$ does not force the range of $\dot F$ to be cofinal in $\On^M$.

Conversely, suppose that $\langle D_i\mid i\in I\rangle\in\C$ with $I\in M$ is a sequence of dense subclasses of $\PP$ and $p\in\PP$ is such that there exist no $q\leq_\PP p$ and $\langle d_i\mid i\in I\rangle\in M$ with each $d_i\subseteq D_i$ predense below $q$. Let $\langle C_\alpha\mid\alpha\in\On^M\rangle$ be a hierarchy on $M$. Now let $G$ be $\PP$-generic over $\MM$ with $p\in G$. In $\MM[G]$, let $F:I\ra\On^M$ be the function defined by $$F(i)=\min\{\alpha\in\On^M\mid C_\alpha\cap D_i\cap G\neq\emptyset\}$$ for $i\in I$.
Using the forcing theorem and Observation \ref{defsubsetclasses}, we may choose a name $\dot F\in\C$ for $F$ and a condition $q\leq_\PP p$ in $G$ such that the above property of $\dot F$ is forced by $q$. But then $q$ forces that $\dot F$ is cofinal in the ordinals -- otherwise we could strengthen $q$ to some $r$ which forces the range of $\dot F$ to be contained in some ordinal $\alpha$ and so $d_i=D_i\cap C_\alpha$ would be predense below $r$ for every $i\in I$, contradicting our assumption. 
\end{proof}

If we additionally assume the existence of a set-like wellorder, we obtain the following strengthening of Lemma \ref{lemma:pretame cofinal function}.  

\begin{lemma} \label{lem:surjection onto Ord from cofinal function} 
Suppose that $\MM=\langle M,\C\rangle$ is a countable transitive model of $\GBC^-$
and that $\PP$ is a notion of class forcing for $\MM$ which satisfies the forcing theorem. 
Assume that $\dot{F}\in \C^\PP$ and $p\in\PP$ are such that $p\Vdash_\PP\anf{\dot{F}\colon\check\kappa\rightarrow \On^M\text{ is cofinal}}$ for some $M$-cardinal $\kappa$.
Then there is a class name $\dot{E}\in \C$ and $q\leq_\PP p$ such that $q\Vdash_\PP\anf{ \dot{E}\colon\check\kappa\rightarrow\On^M\text{ is surjective.}}$ 
\end{lemma} 

\begin{proof}
Since $\PP$ satisfies the forcing theorem, $$A=\{\langle r,\alpha,\beta \rangle\mid \exists s\leq_\PP r\ s\Vdash_\PP \dot{F}(\check\alpha)=\check\beta\}\in \C.$$ 
Hence, making use of a set-like wellorder, there is a sequence $C=\langle C_i\mid i\in\On^M\rangle\in\C$ such that each $C_{i}$ is of the form
$$A_{r,\alpha}=\{\beta\in\On^M\mid \exists s\leq_\PP r\ s\Vdash_\PP \dot{F}(\check\alpha)=\check\beta\}$$ 
for some $r\in\PP$ and $\alpha\in \On^M$ such that $A_{r,\alpha}$ is a proper class, 
and moreover each such class $A_{r,\alpha}$ appears unboundedly often in $C$. 

\begin{claim} 
There is a class $D=\langle D_\beta\mid\beta\in\On^M\rangle$ such that the classes $D_\beta$ form a partition of $\On^M$ and  
$C_\alpha\cap D_\beta$ is a proper class for all $\alpha,\beta\in\On^M$. 
\end{claim} 
\begin{proof} 
Let $k:\On^M\times\On^M\ra\On^M$ be a bijection in $\C$ such that whenever $\bar\gamma<\gamma$,
$k(\beta,\bar\gamma)<k(\beta,\gamma)$. Now we recursively define
sets of ordinals $D_\beta^\gamma\in M$ in the following way: We start with $D_0^0=\emptyset$. Let $\alpha,\beta,\gamma\in\On^M$ be such that $\alpha=k(\beta,\gamma)$ and assume that for 
all $\bar\beta,\bar\gamma$ with $k(\bar\beta,\bar\gamma)<\alpha$, $D_{\bar\beta}^{\bar\gamma}$ has already been defined. 
Now let $D_\beta^\gamma=\bigcup_{\bar\gamma<\gamma}D_\beta^{\bar\gamma}\cup\{\delta\}$, where 
$\delta$ is the least ordinal in $C_\gamma\setminus\bigcup\{D_{\bar\beta}^{\bar\gamma}\mid k(\bar\beta,\bar\gamma)<\alpha\}$.
Finally, put $D_\beta=\bigcup_{\gamma\in\On^M}D_\beta^\gamma$ for each $\beta\in\On^M$. 
By construction, if $\beta\neq\bar\beta$ then $D_\beta$ and $D_{\bar\beta}$ are disjoint. Moreover,
since $C_\alpha$ appears unboundedly often in the enumeration defined above, $C_\alpha\cap D_\beta$
is a proper class for all $\alpha,\beta\in\On^M$. 
\end{proof} 
Suppose that $D$ is a class as in the statement of the previous claim. 
If $G$ is $\PP$-generic over $\MM$ with $p\in G$, let $E\colon \kappa\rightarrow \On^M$ be the function given by $E(\alpha)=\beta$ if $\dot{F}^G(\alpha)\in D_\beta$. 
Since $\PP$ satisfies the forcing theorem, by Observation \ref{defsubsetclasses}, there is a class name $\dot{E}\in\C$ and a condition $q\in G$ below $p$ which forces that $\dot E$ satisfies this definition. 

\begin{claim}
$q\Vdash_\PP\anf{\dot{E}\colon\check\kappa\rightarrow\On^M\text{ is surjective}}.$ 
\end{claim} 
\begin{proof} 
Suppose the contrary. Since $\PP$ satisfies the forcing theorem, there is a condition $r\leq_\PP q$ and some
ordinal $\beta$ such that $r\Vdash_\PP\dot E(\check\alpha)\neq\check\beta$ for all $\alpha<\kappa$. 
Then there is $\alpha<\kappa$ such that $A_{r,\alpha}=\{\gamma\in\On^M\mid\exists s\leq_\PP r\ s\Vdash_\PP\dot F(\check\alpha)=\check\gamma\}$ is a proper class,
since otherwise $r$ forces that the range of $\dot F$ is bounded in $\On^M$, contradicting our assumption on $\dot{F}$. 
By the previous claim, $A_{r,\alpha}\cap D_\beta$ is nonempty. Choose $\gamma\in A_{r,\alpha}\cap D_\beta$ and $s\leq_\PP r$ 
so that $s\Vdash_\PP\dot{F}(\check\alpha)=\check\gamma$. 
Then $s\Vdash_\PP\dot{E}(\check\alpha)=\check\beta$, contradicting our choice of $r$ and of $\beta$. 
\end{proof} 
This completes the proof of Lemma \ref{lem:surjection onto Ord from cofinal function}. 
\end{proof}

\begin{definition}Let $M$ be a model of $\ZF^-$.
 A relation $T\subseteq\Fm\times M$ is a \emph{first-order truth predicate for $M$} if $$\langle\gbr{\varphi},x\rangle\in T ~ \Longleftrightarrow ~ \langle M,\in\rangle\models\varphi(x)$$ holds for every $\gbr{\varphi}\in\Fm$ and every $x\in M$, where $\Fm$ denotes the set of G\"odel codes of $\L_\in$-formulae with one free variable. 
\end{definition} 

Using the previous lemma, we are now ready to prove the main result of this section. 

\begin{theorem}\label{thm:failure ft}
Suppose that $\MM=\langle M,\C\rangle$ is a countable transitive model of $\GBC^-$ such that $\C$ does not contain a first-order truth predicate for $M$.\footnote{Note that by Tarski's theorem on the undefinability of truth, every model of the form $\langle M,\C\rangle\models\GBC^-$, where $\C$ only consists of the definable subsets of $M$, satisfies these requirements.} If $\PP$ is a non-pretame
notion of class forcing for $\MM$, then there is a notion of class forcing $\QQ$ for $\MM$ and a dense embedding $\pi\colon\PP\to\QQ$ in $\C$ such that $\QQ$ does not satisfy the forcing theorem.
\end{theorem}

\begin{proof}
Without loss of generality, we may assume that $\PP$ satisfies the forcing theorem. 
\setcounter{claim}{0}
Using Lemma \ref{lemma:pretame cofinal function}, the axiom of choice, and Lemma \ref{lem:surjection onto Ord from cofinal function}, we can choose $p\in\PP$ and a class name $\dot F$ such that
$$p\Vdash_\PP\anf{\dot F:\check\kappa\ra\check M\text{ is surjective}}$$
for some $M$-cardinal $\kappa$.

We extend $\PP$ to a forcing notion $\QQ$ by adding suprema $p_{\alpha,\beta}$
for the classes $$D_{\alpha,\beta}=\{q\leq_\PP p\mid q\Vdash_\PP\dot F(\check\alpha)\in\dot F(\check\beta)\}$$
for all $\alpha,\beta<\kappa$ such that $D_{\alpha,\beta}$ is nonempty. Let $X=\{\langle\alpha,\beta\rangle\in\kappa^2\mid D_{\alpha,\beta}\neq\emptyset\}$.
The following arguments generalize the proof of \cite[Theorem 1.3]{ClassForcing}.
Define $$\dot E=\{\langle\op(\check\alpha,\check\beta),p_{\alpha,\beta}\rangle\mid\langle\alpha,\beta\rangle\in X\}\in M^\QQ.$$
Assume for a contradiction that $\QQ$ satisfies the forcing theorem. We will use $\dot E$ to show that $\C$ contains a first-order
truth predicate for $M$, contradicting our assumptions.

\begin{claim}Let $G$ be $\QQ$-generic over $\MM$ with $p\in G$ and let $E=\dot E^G$ and $F=\dot F^G$. Then in $M[G]$ it holds that 
$\langle\alpha,\beta\rangle\in E$ if and only if $M\models F(\alpha)\in F(\beta)$.
\end{claim}

\begin{proof}
Let $\alpha,\beta<\kappa$ such that $\langle\alpha,\beta\rangle\in E$. Then $\langle\alpha,\beta\rangle\in X$
and $p_{\alpha,\beta}\in G$. But by definition of $\leq_\QQ$, $p_{\alpha,\beta}\Vdash_\QQ\dot F(\check\alpha)\in\dot F(\check\beta)$
and therefore $F(\alpha)\in F(\beta)$ as desired. Conversely, suppose that $x\in y$ in $M$. 
Since $F$ is surjective, there are $\alpha,\beta<\kappa$ such that $F(\alpha)=x$ and $F(\beta)=y$.
Moreover, there must be $q\in G$ which forces that $\dot F(\check\alpha)\in\dot F(\check\beta)$ and so 
$q\leq_\QQ p_{\alpha,\beta}$. In particular, $p_{\alpha\beta}\in G$ and so $\langle\alpha,\beta\rangle\in E$. 
\end{proof}

The next step will be to translate $\L_\in$-formulae into infinitary quantifier-free formulae in the forcing language of $\QQ$, where $\in$ is translated to $\dot E$. The infinitary language $\L_{\On,0}^{\Vdash}(\QQ,M)$ is built up from the atomic formulae $\check q\in\dot G$, $\sigma\in\tau$ and $\sigma=\tau$ for $q\in\QQ$ and $\sigma,\tau\in M^\PP$, the negation operator and set-sized conjunctions and disjunctions. We denote by $\Fm_{\On,0}^{\Vdash}(\QQ,M)$ the class of G\"odel codes of $\L_{\On,0}^\Vdash(\QQ,M)$-formulae.\footnote{A detailed description of $\L_{\On,0}^{\Vdash}(\QQ,M)$ is given in \cite[Section 5]{ClassForcing}.}

Inductively, we assign to every $\L_\in$-formula $\varphi$ with free variables in $\{v_0,\dots,v_{k-1}\}$ and 
all $\vec\alpha=\alpha_0,\dots,\alpha_{k-1}\in\kappa^{k}$ an $\L_{\On,0}^{\Vdash}(\QQ,M)$-formula in the following way:
\begin{align*}
 (v_i=v_j)_{\vec\alpha}^*&=(\check{\alpha}_i=\check{\alpha}_j)\\
 (v_i\in v_j)_{\vec\alpha}^*&=(\op(\check{\alpha}_i,\check{\alpha}_j)\in\dot E)\\
 (\neg\varphi)_{\vec\alpha}^*&=(\neg\varphi_{\vec\alpha}^*)\\
 (\varphi\vee\psi)_{\vec\alpha}^*&=(\varphi_{\vec\alpha}^*\vee\psi_{\vec\alpha}^*)\\
 (\exists v_k\varphi)_{\vec\alpha}^*&=(\bigvee_{\beta<\kappa}\varphi_{\vec\alpha,\beta}^*).
\end{align*}
Note that, by \cite[Lemma 5.2]{ClassForcing},
if $\QQ$ satisfies the definability lemma for either $\anf{v_0\in v_1}$ or $\anf{v_0=v_1}$, then it satisfies 
the uniform forcing theorem for all $\L_{\On,0}^{\Vdash}(\QQ,M)$-formulae, i.e. $$\{\langle p,\gbr{\varphi}\rangle\in \QQ\times\Fm_{\On,0}^{\Vdash}(\QQ,M)\mid p\Vdash_\QQ^\MM\varphi\}\in\C$$
and $\QQ$ satisfies the truth lemma for every $\L_{\On,0}^{\Vdash}(\QQ,M)$-formula $\varphi$ over $\MM$.
The following claim will thus allow us to define a first-order truth predicate over $M$.

\begin{claim}\label{claim:truth}
 For every $\L_\in$-formula $\varphi$ with free variables among $\{v_0,\dots,v_{k-1}\}$ and for all $\vec x=x_0,\dots,x_{k-1}$ in $M$,
 the following statements are equivalent:
 \begin{enumerate-(1)}
  \item $M\models\varphi(\vec x)$.
  \item $\forall\vec\alpha\in\kappa^k\,\forall q\leq_\PP p\ \left[q\Vdash_\PP\anf{\forall i<k\ \dot F(\check\alpha_i)=\check x_i}\ra q\Vdash_\QQ\varphi_{\vec\alpha}^*\right]$.
  \item $\exists\vec\alpha\in\kappa^k\,\exists q\leq_\PP p\ \left[q\Vdash_\PP\anf{\forall i<k\ \dot F(\check\alpha_i)=\check x_i}\wedge q\Vdash_\QQ\varphi_{\vec\alpha}^*\right]$.
 \end{enumerate-(1)}
\end{claim}

\begin{proof}
 Observe that since $p\Vdash_\PP\anf{\dot F:\check\kappa\ra M\text{ is surjective}}$, (2) always implies (3). 
We argue by induction on the construction of the formula $\varphi$ that (1) implies (2) and that (3) implies (1).

We start with the atomic formula $\anf{v_i\in v_j}$. 
Suppose first that (1) holds, that is $M\models x\in y$. Let $\vec\alpha\in\kappa^k$, let $i,j<k$, and let $q\leq_\PP p$ be such that $q\Vdash_\PP\dot F(\check\alpha_i)=\check x\wedge\dot F(\check\alpha_j)=\check y$.
Take a $\QQ$-generic filter with $q\in G$. Since $q\leq_\QQ p_{\alpha_i,\alpha_j}$, we have $p_{\alpha_i,\alpha_j}\in G$. 
Moreover, $\langle\alpha_i,\alpha_j\rangle\in\dot E^G$, so (2) holds.
Assume now that (3) holds, i.e. there is $\vec\alpha\in\kappa^k$, $i,j<k$ and $q\leq_\PP p$ such that $q\Vdash_\PP\dot F(\check\alpha_i)=\check x\wedge\dot F(\check\alpha_j)=\check y$
and such that $q\Vdash_\QQ(v_i\in v_j)_{\vec\alpha}^*$. Let $G$ be $\QQ$-generic with $q\in G$. Then $\langle\alpha_i,\alpha_j\rangle\in\dot E^G$, and so $p_{\alpha_i,\alpha_j}\in G$. In particular, this means 
that $x=\dot F^G(\alpha_i)\in\dot F^G(\alpha_j)=y$. The proof for $\anf{v_i=v_j}$ is similar.

Next we turn to negations. Suppose first that $M\models\neg\varphi(\vec x)$ and let $\vec\alpha\in\kappa^k$ and $q\leq_\PP p$ with $q\Vdash_\PP\forall i<k\,(\dot F(\check\alpha_i)=\check x_i)$.
Assume, towards a contradiction, that $q\nVdash_\QQ\neg\varphi_{\vec\alpha}^*$. Then there is $r\leq_\QQ q$
with $r\Vdash_\QQ\varphi_{\vec\alpha}^*$. By density, we may assume that $r\in\PP$.
Then $r\leq_\PP p$ and so $\vec\alpha$ and $r$ witness (3) for $\varphi$. By our inductive hypothesis we obtain that $M\models\varphi(\vec x)$,
a contradiction. The implication from (3) to (1) is similar.

Suppose now that $M\models(\varphi\vee\psi)(\vec x)$. Without loss of generality, assume that $M\models\varphi(\vec x)$.
Now if $\vec\alpha\in\kappa^k$ and $q\leq_\PP p$ with $q\Vdash_\PP\forall i<k\,(\dot F(\check\alpha_i)=\check x_i)$, by induction
$q\Vdash_\QQ\varphi_{\vec\alpha}^*$. But then in particular $q\Vdash_\QQ(\varphi\vee\psi)_{\vec\alpha}^*$. 
In order to see that (3) implies (1), suppose that $\vec\alpha\in\kappa^k$ and $q\leq_\PP p$ witness (3). Then there must be a strengthenig $r\in\QQ$ of $q$ which satisfies, without loss of generality, $r\Vdash_\QQ\varphi_{\vec\alpha}^*$.
By density of $\PP$ in $\QQ$, we can assume that $r\in\PP$. This means that $\vec\alpha$ and $r$ witness that (3)
holds for $\varphi$, so $M\models\varphi(\vec x)$. 

We are left with the existential case. Assume first that $M\models\exists v_k\varphi(\vec x)$. Take $y\in M$ such that $M\models\varphi(\vec x,y)$ and 
let $\vec\alpha\in\kappa^k$ and $q\leq_\PP p$ with $q\Vdash_\PP\forall i<k\,(\dot F(\check\alpha_i)=\check x_i)$. Let now 
$G$ be $\QQ$-generic with $q\in G$. By an easy density argument there must be $r\leq_\PP p$ and $\beta<\kappa$
with $r\in G$ and $r\Vdash_\PP\dot F(\check\beta)=\check y$. By induction, $r\Vdash_\QQ\varphi_{\vec\alpha,\beta}^*$.
In particular, $M[G]\models(\exists v_k\varphi)_{\vec\alpha}^*$. The converse follows in a similar way. 
\end{proof}

Let $\Fm_1$ denote the set of all G\"odel codes of $\L_\in$-formulae whose only free variable is $v_0$. 
As a consequence of Claim \ref{claim:truth}, the class
$$T=\{\langle\gbr{\varphi},x\rangle\mid\gbr\varphi\in\Fm_1\wedge x\in M\wedge\forall\alpha<\kappa\,\forall q\leq_\PP p\ q\Vdash_\PP\dot F(\check\alpha)=\check x\ra q\Vdash_\QQ\varphi_\alpha^*\}$$
defines a first-order truth predicate for $M$, contradicting our assumptions on $\MM$. 
\end{proof}

\begin{corollary}\label{cor:pretame densely ft}
Suppose that $\MM=\langle M,\C\rangle$ is a countable transitive model of $\GBC^-$ such that $\C$ does not contain a first-order truth predicate for $M$. 
Then a notion of class forcing $\PP$ for $\MM$ is pretame if and only if it densely satisfies the forcing theorem. \hfill\qedsymbol
\end{corollary}

Furthermore, the proof of Theorem \ref{thm:failure ft} yields the following, with $\PP=\Col(\omega,\On)^M$ as a witness. 

\begin{corollary}
Suppose that $\MM=\langle M,\C\rangle$ is a countable transitive model of $\GBC^-$ such that $\C$ does not contain a first-order truth predicate for $M$. Then there is a notion of class forcing $\PP$ for $\MM$ which satisfies the forcing theorem, such that there is an $\omega$-sequence $S$ of subclasses of $\PP$, for which $\PP_S$ does not satisfy the forcing theorem. \hfill\qedsymbol 
\end{corollary}

Given a notion of class forcing $\PP$, a condition $p\in\PP$, and a first order formula $\varphi$, it is easy to see that (as in the case of set forcing) $p\Vdash_\PP\exists x\varphi(x)$ if and only if the class of all $q\leq_\PP p$ such that there is a $\PP$-name $\sigma$ with $q\Vdash_\PP\varphi(\sigma)$ is dense below $p$. The maximality principle states that it is not necessary to strengthen $p$ in order to obtain a witnessing name for an existential formula. While this is a valid principle for set forcing, we observe that for notions of class forcing which satisfy the forcing theorem, this principle is equivalent to the $\On$-cc over models of $\GBC$.

\begin{definition}\label{def:mp}
A notion of class forcing $\PP$ for $\MM$ which satisfies the forcing theorem is said to satisfy the \emph{maximality principle over $\MM$} if whenever
$p\Vdash_\PP\exists x\varphi(x,\vec\sigma,\vec\Gamma)$ for some $p\in\PP$, some $\L_\in$-formula $\varphi(v_0,\dots, v_m,\vec\Gamma)$ with class name parameters $\vec\Gamma\in(\C^\PP)^n$, and $\vec\sigma$ in $(M^\PP)^m$, then there
is $\tau\in M^\PP$ such that $p\Vdash_\PP\varphi(\tau,\vec\sigma,\vec\Gamma)$.
\end{definition}

\begin{lemma}
Assume that $\MM$ is a model of $\GBC$ and let $\PP$ be a notion of class forcing for $\MM$ which satisfies
the forcing theorem. Then $\PP$ satisfies the maximality principle if and only if it satisfies the $\On$-cc over $\MM$.
\end{lemma}

\begin{proof}
Suppose first that $\PP$ satisfies the maximality principle and let $A\in\C$ be an antichain of $\PP$.
Since $\C$ contains a well-ordering of $M$, we can extend $A$ to a maximal antichain $A'\in\C$. It is enough
to show that $A'\in M$. Clearly, $\one_\PP\Vdash_\PP\exists x\,(x\in\check {A'}\cap\dot G)$. 
Using the maximality principle, we obtain $\sigma\in M^\PP$ such that $\one_\PP\Vdash_\PP\sigma\in\check {A'}\cap \dot G$.
But then since $\rnk(\sigma^G)\leq\rnk(\sigma)$ for every $\PP$-generic filter $G$,
${A'}\subseteq\PP\cap(\V_\alpha)^M$ for $\alpha=\rnk(\sigma)$ and so ${A'}\in M$.

Conversely, assume that $\PP$ satisfies the $\On$-cc over $\MM$ and let $p\Vdash_\PP\exists x\varphi(x,\vec\sigma,\vec\Gamma)$. 
Using the global well-order in $\C$ we can find an antichain $A\in\C$ which is maximal in 
$\{q\leq_\PP p\mid\exists\sigma\in M^\PP\ q\Vdash_\PP\varphi(\sigma,\vec\sigma,\vec\Gamma)\}\in\C$. Note that $\sup A=p$ and that $A\in M$ by assumption.
For every $q\in A$, choose a name $\tau_q\in M^\PP$ such that $q\Vdash_\PP\varphi(\tau_q,\vec\sigma,\vec\Gamma)$.
Furthermore, for every $\mu\in\dom{\tau_q}$, let $A_\mu^q$ be a maximal antichain in 
$\{r\leq_\PP q\mid\exists s\ \langle\mu,s\rangle\in\tau_q\wedge r\leq_\PP s\}$.
Now put $$\sigma=\{\langle\mu,r\rangle\mid \exists q\in A\ \mu\in\dom{\tau_q}\wedge r\in A_\mu^q\}.$$
By construction, $q\Vdash_\PP\sigma=\tau_q$ for every $q\in A$ and so $p\Vdash_\PP\varphi(\sigma)$.
\end{proof}

\section{Preservation of axioms}\label{sec:axioms}

The following theorem adapts \cite[Proposition 2.17 and Lemma 2.19]{MR1780138} to our generalized setting. 

\begin{theorem}[Stanley]\label{thm:pretame axioms}
Let $\MM=\langle M,\C\rangle$ be a countable transitive model of $\GB^-$. Then the following statements hold for every notion of class forcing $\PP$ for $\MM$.
\begin{enumerate-(1)}
 \item If $\PP$ is pretame and $G$ is $\PP$-generic over $\MM$, then $\MM[G]$ satisfies $\GB^-$. Moreover, if $\MM$ is a model of $\GBC^-$, then so is $\MM[G]$, and if $\MM$ has a hierarchy, then so does $\MM[G]$. 
 \item Suppose that $\MM$ has a hierarchy, and that for every $\PP$-generic filter $G$, either Replacement or Collection holds in $\MM[G]$. Then $\PP$ is pretame. 
\end{enumerate-(1)}
\end{theorem}

\begin{proof}
For (1), suppose that $\PP$ is pretame and that $G$ is $\PP$-generic over $\MM$. Note that by Theorem \ref{thm:pretame ft}, $\PP$ satisfies the forcing theorem. It is easy to check that $\MM[G]$ satisfies all set axioms of $\GB^-$ except possibly for Separation, Collection and Union. Moreover, Collection together with Separation implies Replacement, and the preservation of Separation can easily be seen to imply the preservation of Union. 

To see that $\MM[G]$ satisfies Collection, assume that $\MM[G]\models\forall x\in\sigma^G\,\exists y\ \varphi(x,y,\Gamma^G)$, where 
$\sigma\in M^\PP$, $\Gamma\in\C^\PP$ and $\varphi$ is an $\L_\in$-formula with one class parameter.
By the truth lemma there is $p\in G$ such that $p\Vdash_\PP\forall x\in\sigma\,\exists y\,\varphi(x,y,\Gamma)$.
For each $\langle\pi,r\rangle\in\sigma$, the class
$$D_{\pi,r}=\{s\in\PP\mid [s\leq_\PP p,r\wedge\exists\mu\in M^\PP\,(s\Vdash_\PP\varphi(\pi,\mu,\Gamma))]\vee s\bot_\PP r\}\in\C$$ is dense below $p$ in $\PP$. By pretameness there is $q\in G$ which strengthens $p$ and there are sets $d_{\pi,r}\subseteq D_{\pi,r}$ for each $\langle\pi,r\rangle\in\sigma$ such that each $d_{\pi,r}\in M$ is predense below $q$. 
Using Collection in $\MM$, there is a set $x\in M$ such that for each $\langle\pi,r\rangle\in\sigma$ and for each $s\in d_{\pi,r}$ that is compatible with $r$ there is $\mu\in x$ such that $s\Vdash_\PP\varphi(\pi,\mu,\Gamma)$. Now put
$$\tau=\{\langle\mu,s\rangle\mid\mu\in x\wedge\exists\langle\pi,r\rangle\in\sigma\,(s\in d_{\pi,r}\wedge s\Vdash_\PP\varphi(\pi,\mu,\Gamma))\}.$$
By construction, $\MM[G]\models\forall x\in\sigma^G\,\exists y\in\tau^G\,\varphi(x,y,\Gamma^G)$.

Next we turn to Separation. Let $\sigma\in M^\PP,\Gamma\in\C^\PP$ and $\varphi$ an $\L_\in$-formula with one class parameter. We need to find a set name for $\{x\in\sigma^G\mid\varphi(x,\Gamma^G)\}$. For each $\langle\pi,r\rangle\in\sigma$ let 
$$D_{\pi,r}=\{q\leq_\PP r\mid q\Vdash_\PP\varphi(\pi,\Gamma)\}\in\C.$$
By pretameness we can take $p\in G$ and a sequence $\langle d_{\pi,r}\mid\langle\pi,r\rangle\in\sigma\rangle\in M$ such that each $d_{\pi,r}\subseteq D_{\pi,r}$ is predense below $p$. Then 
$$\tau=\{\langle\pi,q\rangle\mid\exists r\,(\langle\pi,r\rangle\in\sigma\wedge q\in d_{\pi,r}\wedge q\Vdash_\PP\varphi(\pi,\Gamma))\}$$
names the set $\{x\in\sigma^G\mid\varphi(x,\Gamma^G)\}$.

The class axioms of extensionality and foundation are trivial consequences of elements of $\C[G]$ being subsets of $M[G]$, and of $\C$ being an element of $\V\models\ZFC$.
To see that $\MM[G]$ satisfies first-order class comprehension, note that $$\Gamma=\{\langle\sigma,p\rangle\mid p\Vdash_\PP\varphi(\sigma,\Gamma_0,\dots,\Gamma_{n-1})\}\in\C^\PP$$ is a class name for the class $\{x\mid\varphi(x,\Gamma_0^G,\dots,\Gamma_{n-1}^G)\}$.

If $\prec$ is a set-like well-order of $M$ in $\C$ then $$x\vartriangleleft y\Llr\exists\sigma\in M^\PP[x=\sigma^G\wedge\forall\tau\in M^\PP(y=\tau^G\ra\sigma\prec\tau)]$$
defines a set-like well-order of $M[G]$ in $\C[G]$.
Finally, if $\langle C_\alpha\mid\alpha\in\On^M\rangle$ is a hierarchy for $\MM$, we can define a hierarchy
 $\langle D_\alpha\mid \alpha\in\On^M\rangle$ in $\C[G]$ by
 $$D_\alpha=\{x\in\MM[G]\mid\exists\sigma\in M^\PP\cap C_\alpha\ \sigma^G=x\}=\{\langle\sigma,\one\rangle\mid\sigma\in C_\alpha\}^G\in M[G]$$
for every $\alpha\in\On^M$.

Now we turn to (2). Suppose that $\langle C_\alpha\mid\alpha\in\On^M\rangle$ witnesses that $\MM$ has a hierarchy. Assume, towards a contradiction, that $\langle D_i\mid i\in I\rangle$ is a sequence of dense classes 
and $p$ is a condition in $\PP$ which witness that pretameness fails. By our assumption there is a $\PP$-generic filter $G$ containing 
$p$ such that $\MM[G]$ satisfies either Replacement or Collection. Now consider the function
$$F:I\ra\On^M,F(i)=\min\{\alpha\in\On^M\mid G\cap D_i\cap C_\alpha\neq\emptyset\}.$$
By assumption, there is a set $x\in M[G]$ with $\ran F\subseteq x$. 
Let $\gamma$ be the supremum of all ordinals in $x$ and let
$$D=\{q\leq_\PP p\mid\exists i\in I\,\forall r\in D_i\cap C_\gamma\,(q\bot_\PP r)\}.$$
By assumption, $D$ is dense below $p$. Pick $q\in G\cap D$ and let $i\in I$ such that $q$ is incompatible with all elements of $D_i\cap C_\gamma$. But then $F(i)>\gamma$, a contradiction.
\end{proof}

Our next result provides a partial answer to \cite[Question 10.1]{ClassForcing}, that is, over countable transitive models of $\GBC^-$, we show that Separation implies Collection, and hence also Replacement, in generic extensions for notions of class forcing that satisfy the forcing theorem. This will be an easy consequence of the following.

\begin{lemma}\label{sepreplemma}
  Suppose that $\MM=\langle M,\C\rangle$ is a countable transitive model of $\GBC^-$. Let $\PP\in\C$ be a notion of class forcing which satisfies the forcing theorem, and let $G$ be $\PP$-generic over $\MM$. 
Suppose that $\C[G]$ contains a cofinal function $F\colon\lambda\rightarrow \On^M$, for some $\lambda\in \On^M$. Then Separation fails in $\MM[G]$.
\end{lemma}
\begin{proof}
Choose a set-like wellorder $\prec\,\in\C$ of $M$. We pick a class name $\dot F\in\C^\PP$ and $r\in G$ such that $\dot F^G=F$, and such that 
$$r\Vdash_\PP \dot{F}\colon\check\kappa\rightarrow \On^M\text{ is cofinal}.$$ 
As $r$ plays no role in the proof to come, we may assume that $r=\one_\PP$. 
Let $C\subseteq \On^M\times\On^M$ be an element of $\C$ such that each $C_{\gamma}=\{\delta\in \On^M\mid \langle\gamma,\delta\rangle\in C\}$ for $\gamma\in \On^M$ is either of the form 
$$A_{p,\alpha}=\{\beta\in\On^M\mid \exists q\leq_\PP p\,[ q\Vdash_{\mathbb{P}} \dot{F}(\check\alpha)=\check\beta]\}\in\C$$ 
for $p\in \mathbb{P}$ and $\alpha\in \On^M$ such that $A_{p,\alpha}$ is a proper class of $\MM$, or of the form
$$B_{p,\alpha,\tau}=\{\beta\in\On^M\mid\exists q\leq_\PP p\,[ q\Vdash_{\mathbb{P}}(\dot{F}(\check\alpha)=\check\beta\,\wedge\,\check\alpha\in \tau)]\}\in \C$$ 
for $p\in \mathbb{P}$, $\alpha\in \On^M$ and $\tau\in M^\PP$ such that $B_{p,\alpha,\tau}$ is a proper class of $\MM$, 
and moreover each such $A_{p,\alpha}$ and $B_{p,\alpha,\tau}$ appears unboundedly often in the enumeration $\langle C_\gamma\mid \gamma\in\On^M\rangle$.

\begin{claim*} 
There is $D\in\C$ such that both $C_\gamma\cap D$ and $C_\gamma\setminus D$ are proper classes of $\MM$ for all $\gamma\in \mathrm{Ord}^M$.  
\end{claim*} 
\begin{proof} 
This can be achieved by recursively defining $D_\gamma$ and $D'_\gamma$ in $M$ as follows. Let $D_0=D'_0=\emptyset$. 
Suppose that  $D_\gamma$ and $D'_\gamma$ have already been defined. 
Let $D_{\gamma+1}\in M$ be the set obtained from $D_\gamma$ by adding the $\prec$-least $x\in C_\gamma\setminus(D_\gamma\cup D'_\gamma)$. 
Let $D'_{\gamma+1}\in M$ be the set obtained from $D'_\gamma$ by adding the $\prec$-least $x\in C_\gamma\setminus(D_{\gamma+1}\cup D'_\gamma)$. 
If $\gamma$ is a limit ordinal, let $D_{\gamma}=\bigcup_{\delta<\gamma}D_\delta$ and $D'_\gamma=\bigcup_{\delta<\gamma}D'_\delta$. Let $D=\bigcup_{\gamma\in\On^M}D_\gamma\in\C$. 
Since each $C_\gamma$ appears unboundedly often in the enumeration, both $C_\gamma\cap D$ and $C_\gamma\setminus D$ are proper classes. 
\end{proof}
Suppose that $D\in\C$ is as provided by the claim. 
Since $\PP$ satisfies the forcing theorem, to show that Separation fails in $\MM[G]$, 
it suffices to find $\Gamma\in\C^{\PP}$ such that $\Gamma^G\in\E$ is a subset of some $s\in M[G]$ and such that there is no $p\in G$ and no $\tau\in M^\PP$ with $p\Vdash_\PP\Gamma=\tau$.
Let $\Gamma\in\C$ be a name for $\{\alpha<\kappa\mid F(\alpha)\in D\}$, which is an element of $\E$, as the latter is closed under first order definability.
The above set $s$ will be equal to $\kappa$. 
Suppose for a contradiction that $p\in G$ is such that $p\Vdash_{\mathbb{P}}\Gamma=\tau$ for some $\tau\in M^\PP$.
We first claim that there is an $\alpha<\kappa$ such that $A_{p,\alpha}$ is a proper class.
Assume that such an $\alpha$ does not exist. Then $p$ forces that the range of $\dot F$ is bounded in $\On^M$, contradicting our assumption on $\dot{F}$.
The above claim implies that $A_{p,\alpha}\cap D=\{\beta\mid\exists q\leq_\PP p\,[ q\Vdash_{\PP}(\dot F(\check\alpha)=\check\beta\,\land\,\check\alpha\in\tau)]\}=B_{p,\alpha,\tau}$ is a proper class.
Then $B_{p,\alpha,\tau}\setminus D$ is empty, contradicting the choice of $D$. 
\end{proof}

\begin{theorem}\label{thm:sep rep}
Suppose that $\MM=\langle M,\C \rangle$ is a countable transitive model of $\GBC^-$. 
Let $\PP\in\C$ be a notion of class forcing which satisfies the forcing theorem, and let $G$ be $\PP$-generic over $\MM$.
If $\MM[G]$ satisfies Separation, then $\MM[G]$ satisfies Collection, and hence also Replacement.
\end{theorem}
\begin{proof}
Assume that Collection fails in $\MM[G]$. By Theorem \ref{thm:pretame axioms}, (1), this implies that $\PP$ is not pretame. We want to show that Separation fails in $\MM[G]$. Using Lemma \ref{lemma:pretame cofinal function} together with the axiom of choice (which is a consequence of the existence of a set-like wellorder), we can pick $\lambda\in\On^M$ and a class $F\in\C[G]$ that is a cofinal function from $\lambda$ to $\On^M$. The statement of the theorem now follows directly from Lemma \ref{sepreplemma}
\end{proof}

\begin{corollary}\label{cor:pretame axioms}
 Suppose that $\MM=\langle M,\C\rangle$ is a countable transitive model of $\GBC^-$, and let $\PP$ be a notion of class forcing for $\MM$. Then the following statements are equivalent:
 \begin{enumerate-(1)}
  \item $\PP$ is pretame.
  \item $\PP$ preserves $\GB^-$.
  \item $\PP$ preserves Collection.
  \item $\PP$ preserves Replacement.
  \item $\PP$ preserves Separation and satisfies the forcing theorem. 
 \end{enumerate-(1)}
\end{corollary}

\begin{proof}
 The equivalence of (1)--(4) follows from Theorem \ref{thm:pretame axioms}.
 The implication from (1) to (5) follows from Theorem \ref{thm:pretame ft} together with Theorem \ref{thm:pretame axioms}. 
The implication from (5) to (3) is provided by Theorem \ref{thm:sep rep}.
\end{proof}

\section{Boolean completions}\label{sec:bc}
 
As has been shown in \cite{ClassForcing}, the existence of Boolean completions is closely related to the forcing theorem. Namely by \cite[Theorem 5.5]{ClassForcing}, if $\MM$ is a countable transitive model of $\GB^-$with a hierarchy, and $\PP$ is a separative notion of class forcing for $\MM$, then $\PP$ has a Boolean $M$-completion if and only if it satisfies the forcing theorem for all $\L_\in$-formulae. Thus the following is a consequence of Theorem \ref{cor:pretame densely ft}.

\begin{theorem}\label{thm:ft bc}
Suppose that $\MM$ is a countable transitive model of $\GBC^-$ such that $\C$ does not contain a first-order truth predicate for $M$. Then a notion of class forcing $\PP$ for $\MM$ is pretame over $\MM$ if and only if it densely has a Boolean $M$-completion. \hfill\qedsymbol
\end{theorem}

\begin{lemma}\label{lemma:unique bc}
Let $\MM$ be a countable transitive model of $\GB^-$.
If a notion of class forcing $\PP$ for $\MM$ has a Boolean $\C$-completion $\BB$, then it is unique. 
Moreover, if $\PP$ has a unique Boolean $M$-completion $\BB$, then $\BB$ is a Boolean $\C$-completion of $\PP$. 
\end{lemma}

\begin{proof}
The proof of the first statement is exactly as for set forcing. Suppose now that $\BB$ is the unique Boolean $M$-completion of $\PP$ and suppose for a contradiction that $A\subseteq\BB$ is a class in $\C$ which does not have a supremum in $\BB$. Let $\QQ$ be the forcing notion obtained from $\PP$ by adding $\sup A$. Then since $\PP$ satisfies the forcing theorem, by Lemma \ref{lemma:add suprema ft} so does $\QQ$ and hence $\QQ$ has a Boolean $M$-completion $\BB'$. But by our assumption, $\BB$ and $\BB'$ are isomorphic and hence $\sup A$ exists in $\BB$, a contradiction.
\end{proof}

\begin{definition}Suppose that $\PP$ is a notion of class forcing for $\MM=\langle M,\C\rangle\models\GB^-$. 
 If $A,B\subseteq\PP$ with $A,B\in\C$, we say that $\sup_\PP A=\sup_\PP B$ if 
\begin{enumerate-(1)}
 \item $A$ is predense below every $b\in B$ and
 \item $B$ is predense below every $a\in A$. 
\end{enumerate-(1)}
Note that this definition is possible even if the suprema do not exist in $\PP$. On the other hand, if $\sup A=\sup B$ and $A$ has a supremum in $\PP$ then so does $B$ and indeed they coincide. 
\end{definition}

The following observation is a slight strengthening of a lemma which is
essentially due to Joel Hamkins and appears within the proof of \cite[Theorem 9.4]{ClassForcing}.
The below proof is very similar to the one appearing in \cite[Theorem 9.4]{ClassForcing}, however we also improved our original presentation.

\begin{lemma}\label{lem:no sup}
Suppose that $\MM$ is a countable transitive model of $\GBC^-$. 
If $\PP$ does not satisfy the $\On$-cc, then there is an antichain $A\in\C$ such that for every $B\in M$ with $B\subseteq\PP$, 
$\sup B\neq\sup A$. In particular, $A$ does not have a supremum in $\PP$.  
\end{lemma}

\begin{proof}Let $A\in\C$ be a class-sized antichain in $\PP$. We claim that there is a subclass of $A$ in $\C$ which fulfills the desired properties. Suppose for a contradiction that no such subclass exists. 
Using the set-like well-order of $M$, we can assume that the domain of $\PP$ is $\On^M$. Let $\pi:\On^M\ra A$ be a bijection in $\C$. Furthermore, there is an injection $\varphi:\mathcal P(\On^M)\cap M\ra\On^M$ in $\C$. This gives us a mapping $i:\mathcal P(\On^M)\cap\C\ra\On^M$ in $\V$ which maps $X\subseteq\On^M$ to $\varphi(B)$, where $B$ is the least (with respect to our given global well-order) set $B\subseteq\PP$ in $M$ such that $\sup_\PP \pi''X=\sup_\PP B$. Since $A$ is an antichain, $i$ is injective. 
Moreover, whether $i(X)=\alpha$ is definable over $\MM$, so
$$C=\{\alpha\in\On^M\mid\pi(\alpha)\nleq_\PP\alpha\wedge i(X_\alpha)=\alpha\}$$
is in $\C$ for $X_\alpha=\{\beta\in\On^M\mid\pi(\beta)\leq_\PP\alpha\}$.
\setcounter{claim}{0}
\begin{claim}\label{claim:C}For each $\alpha\in\On^M$ we have $\alpha\in C$ if and only if there is $X\in\mathcal P(\On^M)\cap\C$ such that $i(X)=\alpha$ and $\alpha\notin X$.
\end{claim}
\begin{proof}
Suppose first that $\alpha\in C$. Then $\alpha\notin X_\alpha$ and so we can choose $X=X_\alpha$. Conversely, suppose that $X\in\mathcal P(\On^M)\cap\C$ is such that $i(X)=\alpha$ and $\alpha\notin X$. Then $X=X_\alpha$, because $\pi''X$ and $\pi''X_\alpha$ are both subsets of the antichain $A$ and have the same supremum. Hence $\alpha\in C$. 
\end{proof}
We will use Claim \ref{claim:C} to derive a contradiction similar to Russell's paradox. 
Consider $\beta=i(C)$. If $\beta\in C$ then by Claim \ref{claim:C} there is $X$ such that $i(X)=\beta$ but $\beta\notin X$. By injectivity of $i$, this means that $X=C$, a contradiction. On the other hand, it is also impossible that $\beta\notin C$, since otherwise $X=C$ would witness that $\beta\in C$. 
\end{proof}

The following theorem characterizes the $\On$-cc in terms of the existence of Boolean completions. Note that the equivalence of (1) and (2) is exactly the statement of \cite[Theorem 9.4]{ClassForcing}. For the benefit of the reader, we nevertheless give a full proof. 

\begin{theorem}
Suppose that $\MM=\langle M,\C\rangle$ is a countable transitive model of $\GBC^-$. Then the following statements are equivalent for every separative partial order $\PP$:
\begin{enumerate-(1)}
 \item $\PP$ satisfies the $\On$-cc.
 \item $\PP$ has a unique Boolean $M$-completion.
 \item $\PP$ has a Boolean $\C$-completion.
\end{enumerate-(1)}
\end{theorem}

\begin{proof}
Suppose first that $\PP$ satisfies the $\On$-cc. Then $\PP$ is pretame and hence it has a Boolean $M$-completion $\BB$ by Theorem \ref{thm:ft bc}. Assume that $\BB'$ is another Boolean $M$-completion. Without loss of generality, we may assume that the domain of $\PP$ is a subset of the domains of $\BB$ and $\BB'$. Then we can define an isomorphism between $\BB$ and $\BB'$ by mapping $b\in\BB$ to $\sup_{\BB'}A$, where $A\in M$ is a maximal antichain in $\{p\in\PP\mid p\leq_{\BB}b\}$. 
The equivalence of (2) and (3) is a direct consequence of Lemma \ref{lemma:unique bc}. To see that (3) implies (1), suppose that $\BB$ is a Boolean $\C$-completion of $\PP$. Assume, towards a contradiction, that $\PP$ does not satisfy the $\On$-cc. Then neither does $\BB$. But then by Lemma \ref{lem:no sup}, $\BB$ cannot be $\C$-complete. 
\end{proof}

\section{The extension maximality principle}\label{sec:EMP}

This section is motivated by the following easy observation which is mentioned in \cite[Corollary 2.3]{ClassForcing}.
The collapse forcing $\Col_*(\omega,\On)^M$, which consists of functions $n\ra\On^M$ for $n\in\omega$, ordered by reverse inclusion, is dense 
in $\Col(\omega,\On)^M$. However, unlike $\Col(\omega,\On)^M$, which collapses all $M$-cardinals, $\Col_*(\omega,\On)^M$
does not add any new sets, so $\Col(\omega,\On)^M$ and $\Col_*(\omega,\On)^M$ do not have the same generic extensions.
We will show that, under sufficient conditions on the ground model $\MM$, the property of $\PP$ of having the same generic extensions
as all forcing notions into which $\PP$ densely embeds is in fact equivalent to the pretameness of $\PP$. Let $\MM=\langle M,\C\rangle$ be countable transitive model of $\GB^-$.

\begin{definition}
 A notion of class forcing $\PP$ for $\MM$ satisfies the 
 \begin{enumerate-(1)}
    \item  \emph{extension maximality principle (EMP)} over $\MM$ if whenever $\QQ$ is a notion of class forcing for $\MM$ and  $\pi\colon\PP\to\QQ$ is a dense embedding in $\C$ then for every $\QQ$-generic filter $G$ over $\MM$, $M[G]=M[\pi^{-1}[G]]$.
 \item \emph{strong extension maximality principle (SEMP)} over $\MM$ if whenever $\QQ$ is a notion of class forcing for $\MM$, $\pi\colon\PP\to\QQ$ is a dense embedding in $\C$ and $\sigma\in M^\QQ$, then there is $\tau\in M^\PP$ with $\one_{\QQ}\Vdash_\QQ\sigma=\pi(\tau)$.
 \end{enumerate-(1)}
\end{definition}

\begin{definition}
  Given a notion of class forcing $\PP$ and a $\PP$-name $\sigma$, we define the conditions appearing in (the transitive closure of) $\sigma$ by induction on name rank as
\[\tc(\sigma)=\bigcup\{\{p\}\cup\tc(\tau)\mid\langle\tau,p\rangle\in\sigma\}.\]
\end{definition}

\begin{theorem}\label{thm:EMP}Let $\MM=\langle M,\C\rangle$ be a countable transitive model of $\GBC^-$. 
 A notion of class forcing $\PP$ for $\MM$ is pretame if and only if it satisfies the forcing theorem and the EMP. 
\end{theorem}
\begin{proof}
 Suppose first that $\PP$ is pretame. By Theorem \ref{thm:pretame ft}, $\PP$ satisfies the forcing theorem. Let $\QQ$ be a notion of class forcing such that $\PP$ embeds densely into $\QQ$
 and let $G$ be $\QQ$-generic over $\MM$. Without loss of generality, we may assume that $\PP$ is a dense suborder of $\QQ$. Fix a $\QQ$-name $\sigma$. We claim that 
 $\sigma^G\in M[G\cap\PP]$. For every $q\in\tc(\sigma)\cap\QQ$, let $D_q=\{p\in\PP\mid p\leq_\QQ q\vee p\bot_\QQ q\}$.
 Then $D_q$ is a dense subclass of $\PP$. By pretameness, there is $p\in G\cap\PP$ and there are $d_q\subseteq D_q$ which are predense
 below $p$ in $\PP$. Now we define inductively, for every name $\tau$ in $\tc(\{\sigma\})\cap M^\QQ$,
 $$\bar\tau=\{\langle\bar\mu,r\rangle\mid\exists s[\langle\mu,s\rangle\in\tau\wedge r\in d_s\wedge r\leq_\QQ s]\}.$$
 But then $\bar\sigma\in M^\PP$ and $\sigma^G=\bar\sigma^{G\cap\PP}\in M[G\cap\PP]$.
 
 Conversely, assume that $\PP$ is not pretame but satisfies the forcing theorem. Then there is a $\PP$-generic filter $G$ such that Replacement fails in the generic extension $\MM[G]$, and by Theorem \ref{thm:sep rep}, so does Separation (note that this is where we use the existence of a set-like well-order). Hence there is a set $x\in M[G]$ and a first-order formula $\varphi$ with a class parameter $C$ such that $\{y\in x\mid \varphi(y,C)\}\notin M[G]$. By Observation \ref{defsubsetclasses}, $C$ has a class name $\Gamma\in\C^\PP$. Using the truth lemma, we can find $p\in G$ such that
 $p\Vdash_\PP\Gamma\subseteq\sigma$, where $x=\sigma^G$. By assumption, there
are no $q\in G$ and $\tau\in M^\PP$ 
 such that $q\Vdash_\PP\Gamma=\tau$. For $\mu\in\dom{\sigma}$ consider $$A_\mu=\{q\in\PP\mid q\Vdash_\PP\mu\in\Gamma\}.$$
 Let $\QQ$ be the forcing notion
 obtained from $\PP$ by adding $\sup A_\mu$ for each $\mu\in\dom{\sigma}$ such that $A_\mu$ is nonempty below $p$, as described in Section \ref{sec:intro}. 
Without loss of generality, we may assume that $\PP$ is a subset of $\QQ$ and then $\PP$ is actually a dense subset of $\QQ$. Consider the $\QQ$-name
$$\tau=\{\langle\mu,\sup A_\mu\rangle\mid\mu\in\dom{\sigma},A_\mu\text{ is nonempty below }p\}\in M^\QQ.$$
Let $H$ be the $\QQ$-generic induced by $G$, that is the upwards closure of $G$ in $\QQ$. Then $\tau$ is a $\QQ$-name for $\Gamma^G$, so $\tau^H=\Gamma^G\in M[H]\setminus M[G]$, yielding the failure of the EMP.
 \end{proof}
 
 Note that in the proof of Theorem \ref{thm:EMP} we have only used the	 set-like well-order of $M$ to show that every forcing notion which satisfies the forcing theorem and the EMP is pretame; for the other direction, it suffices to assume that $\MM\models\GB^-$ has a hierarchy. 
 
\begin{example}
\emph{Jensen coding} $\PP$ (see \cite{MR645538}) is a pretame notion of class forcing which over a model $M$ of $\ZFC$ adds a generic real $x$ such that the $\PP$-generic extension is of the form $\mathsf L[x]$. Moreover, there is a class name $\Gamma$ for $x$ such that $\one_\PP\Vdash_\PP M[\dot G]=\mathsf L[\Gamma]$, but there is no set name $\sigma$ such that $\one_\PP\Vdash_\PP\sigma=\Gamma$. Let $\QQ$ be the forcing notion obtained from Jensen coding by adding the suprema $p_n=\sup\{p\in\PP\mid p\Vdash_\PP\check n\in\Gamma\}$. Since $\PP$ is pretame and dense in $\QQ$, it follows that $\QQ$ is also pretame. By Theorem \ref{thm:EMP}, $\PP$ satisfies the EMP and hence $\PP$ and $\QQ$ produce the same generic extensions. In particular, this means that if $G$ is $\QQ$-generic then $M[G\cap\PP]=M[G]=\mathsf L[\sigma^G]$, where $\sigma=\{\langle\check n,p_n\rangle\mid n\in\omega\}\in M^\QQ$.
\end{example}

\begin{lemma}\label{lemma:SEMP}
Suppose that $\MM=\langle M,\C\rangle$ is a countable transitive model of $\GBC^-$. Then a notion of class forcing $\PP$ for $\MM$ satisfies the SEMP if and only if it satisfies the $\On$-cc over $\MM$. 
\end{lemma}

\begin{proof}
Suppose first that $\PP$ satisfies the $\On$-cc. Suppose that there is a dense embedding from $\PP$ into some forcing notion $\QQ$. Without loss of generality, we may assume that $\PP$ is a suborder of $\QQ$. We prove by induction on $\rank(\sigma)$
that for every $\sigma\in M^\QQ$ there is $\bar\sigma\in M^\PP$ with $\one_\QQ\Vdash_\QQ\sigma=\bar\sigma$. 
Assume that this holds for all $\tau$ of rank less than $\rank(\sigma)$. Then for every $\tau\in\dom{\sigma}$
there is $\bar\tau\in M^\PP$ with $\one_\QQ\Vdash_\QQ\tau=\bar\tau$. For each condition $q\in\range(\sigma)$, let
$D_q=\{p\in\PP\mid p\leq_\QQ q\}$ and choose an antichain $A_q$ which is maximal in $D_q$.
By assumption, we may do this so that $\langle\langle q,A_q\rangle\mid q\in\range(\sigma)\rangle\in M$. Then put
$$\bar\sigma=\{\langle\bar\tau,p\rangle\mid\exists q\in\QQ\,[\langle\tau,q\rangle\in\sigma\wedge p\in A_q]\}\in M^\PP.$$
By construction, $\one_\QQ\Vdash_\QQ\sigma=\bar\sigma$.

Conversely, suppose that $\PP$ does not satisfy the $\On$-cc. 
Then by Lemma \ref{lem:no sup} there is an antichain $A\in\C$ such that for no $B\in M$ with $B\subseteq\PP$, $\sup_\PP A=\sup_\PP B$. Let $\QQ=\PP\cup\{\sup A\}$ be the extension of $\PP$ given by adding the supremum of $A$. 
Now consider $\sigma=\{\langle\check0,\sup A\rangle\}\in M^\QQ$. We claim that $\sigma$ witnesses the failure of the SEMP. Suppose for a contradiction that there is $\tau\in M^\PP$ such that $\one_\QQ\Vdash_\QQ\sigma=\tau$. Let $\tau=\{\langle\mu_i,p_i\rangle\mid i\in I\}$ for some $I\in M$. But then it is easy to check that $\sup_\PP\{p_i\mid i\in I\}=\sup_\PP A$, contradicting our assumption on $A$. 
\end{proof}

\section{Nice Names}\label{sec:niceness}

This section is motivated by the observation in Lemma \ref{lem:col no nice name} below, namely that - unlike in the context of set forcing -
there are sets of ordinals in class-generic extensions which do not have a nice name. We will characterize both pretameness and the $\On$-cc in terms of the existence of nice names for sets of ordinals. We assume that $\MM=\langle M,\C\rangle$ is a countable transitive model of $\GB^-$.

\begin{definition}
Let $\PP$ be a notion of class forcing. A name $\sigma\in M^\PP$ for a set of ordinals is a \emph{nice name}
if it is of the form $\bigcup_{\alpha<\gamma}\{\check\alpha\}\times A_\alpha$ for some $\gamma\in\On^M$, where each $A_\alpha\in M$ is a set-sized antichain of conditions in $\PP$. 
\end{definition}

\begin{lemma}\label{lem:col no nice name}
Let $\PP$ denote $\Col(\omega,\On)^M$. Then in every
$\PP$-generic extension there is a subset of $\omega$ which does not have a nice $\PP$-name. 
\end{lemma}

\begin{proof}
Consider the canonical name $\sigma=\{\langle\check n,\{\langle n,0\rangle\}\rangle\mid n\in\omega\}$
for the set of natural numbers which are mapped to 0 by the generic function from $\omega$ to the ordinals. 
Let $G$ be $\PP$-generic over $\MM$.
We show that the complement of $\sigma^G$ is an element of $M[G]$, but does not have a nice $\PP$-name in $M$.
Suppose for a contradiction that there are $p\in G$ and a nice $\PP$-name $$\tau=\bigcup_{n\in\omega}\{\check n\}\times A_n\in M,$$
where each $A_n\in M$ is an antichain, such that $p\Vdash_\PP\check\omega\setminus\sigma=\tau$. 
Let $n\in\omega\setminus\dom{p}$ and choose $\alpha>\sup\{r(i)\mid r\in A_n\wedge i\in\dom{r}\}$. Then $q=p\cup\{\langle n,\alpha\rangle\}$ strengthens $p$ and $q\Vdash_\PP\check n\in\tau$. Hence there must be some $r\in A_n$ which is compatible with $q$. But then $n\notin\dom{r}$, so $p$ and $r\cup\{\langle n,0\rangle\}$ are compatible. Let $s\leq_\PP p,r\cup\{\langle n,0\rangle\}$ witness this. Then $s\Vdash_\PP\check n\in\sigma\cap\tau$, a contradiction.

It remains to show that the complement of $\sigma^G$ has a $\PP$-name in $M$, and is thus an element of $M[G]$. For each $n\in\omega$, consider the $\PP$-name $$\tau_n=\check n\cup\{\langle\check m,\{\langle i,0\rangle\mid n\leq i<m\}\rangle \mid m>n\}.$$
Then each $\tau_n$ is a name for the least $k\geq n$ such that $k\notin\sigma^G$. Now put $\tau=\{\langle\tau_n,\one_\PP\rangle\mid n\in\omega\}$. Since by an easy density argument the complement of $\sigma^G$ is unbounded in $\omega$, $\tau$ is as desired. 
\end{proof}

\begin{definition}
A notion of class forcing $\PP$ for $\MM$ is said to be 
\begin{enumerate-(1)}
 \item \emph{nice} if for every $\gamma\in\On^M$, for every $\sigma\in M^\PP$ and for every $\PP$-generic filter $G$ such that $\sigma^G\subseteq\gamma$ there is a nice name $\tau\in M^\PP$ such that $\sigma^G=\tau^G$. 
 \item \emph{very nice} if for every $\gamma\in\On^M$ and for every $\sigma\in M^\PP$ such that 
 $\one_\PP\Vdash_\PP\sigma\subseteq\check\gamma$ there is a nice name $\tau\in M^\PP$ such that $\one_\PP\Vdash_\PP\sigma=\tau$.
\end{enumerate-(1)}
\end{definition}

\begin{example}\label{example:nice}
 \begin{enumerate-(1)}
  \item By Lemma \ref{lem:col no nice name}, $\Col(\omega,\On)^M$ is not nice.
  \item Suppose $M$ satisfies the axiom of choice. Then every pretame notion of class forcing $\PP$ for $\MM$ is nice.
  To see this, let $\gamma\in\On^M$ be an ordinal and let $p\in\PP$ and $\sigma\in M^\PP$ be such that $p\Vdash_\PP\sigma\subseteq\check\gamma$. For each $\alpha<\gamma$, consider the class 
$$D_\alpha=\{q\leq_\PP p\mid q\Vdash_\PP\alpha\in\sigma\vee q\Vdash_\PP\alpha\notin\sigma\}\in\C,$$
which is dense below $p$. By pretameness there exist $q\leq_\PP p$ and a sequence $\langle d_\alpha\mid\alpha<\gamma\rangle\in M$ such that 
for every $\alpha<\gamma$, $d_\alpha\subseteq D_\alpha$ is predense below $q$. For each $\alpha<\gamma$, choose 
an antichain $a_\alpha\subseteq d_\alpha$ which is maximal in $d_\alpha$, and let $A_\alpha=\{r\in a_\alpha\mid r\Vdash_\PP\check\alpha\in\sigma\}$. Then  
$$\tau=\bigcup_{\alpha<\gamma}\{\check\alpha\}\times A_\alpha\in M^\PP$$
is a nice name for a subset of $\gamma$ and $q\Vdash_\PP\sigma=\tau$.
  \item If $\MM\models\GBC^-$, then every notion of class forcing $\PP$ for $\MM$ which satisfies the $\On$-cc is very nice:
  Suppose that $\one_\PP\Vdash_\PP\sigma\subseteq\check\gamma$. For every $\alpha<\gamma$, we can choose an
  antichain $A_\alpha$ which is maximal in $\{q\in\PP\mid \exists \langle\mu,p\rangle\in\sigma\,[q\leq_\PP p\wedge q\Vdash_\PP\mu=\check\alpha]\}$. 
  Since $\PP$ satisfies the $\On$-cc, making use of the global well-order we can do this so that $\langle\langle\alpha,A_\alpha\rangle\mid\alpha<\gamma\rangle\in M$. Then 
  $$\tau=\bigcup_{\alpha<\gamma}\{\check\alpha\}\times A_\alpha\in M^\PP$$
  is a nice name and it is easy to check that $\one_\PP\Vdash_\PP\sigma=\tau$. 
  \item Suppose that $\MM$ has a hierarchy. Then every $M$-complete Boolean algebra $\BB$ is very nice, since we can always define \emph{Boolean values} $\llbracket\varphi\rrbracket_\BB$ for quantifier-free infinitary formulae $\varphi$ which mention only set names (see \cite[Theorem 5.5]{ClassForcing}). More precisely, if 
   $\sigma\in M^\BB$ such that
  $\one_\BB\Vdash_\BB\sigma\subseteq\check{\gamma}$ for some ordinal $\gamma$, the name  $$\tau=\{\langle\check\alpha,\llbracket\check\alpha\in\sigma\rrbracket_\BB\rangle\mid\alpha<\gamma\}\in M^\BB$$
 is a nice name so that $\one_\BB\Vdash_\BB\sigma=\tau$. In particular, this shows that there are very nice notions of class forcing which are not pretame (for example the Boolean $M$-completion of $\Col(\omega,\On)^M$),
 since every notion of class forcing for $\MM$ which satisfies the forcing theorem has a Boolean $M$-completion by \cite[Theorem 5.5]{ClassForcing}.
 \end{enumerate-(1)}
\end{example}

We will need the following.

\begin{lemma}\label{lemma:intersection}
Let $\PP$ be a notion of class forcing for $\MM$. Let $\alpha\in\On^M$ be an ordinal and let $\sigma,\tau\in M^\PP$ be
nice names for subsets of $\alpha$. If $G$ is $\PP$-generic over $\MM$, then there is a $\PP$-name $\mu\in M^\PP$ such that 
$\mu^G=\sigma^G\cap\tau^G$. 
\end{lemma}
\begin{proof}
Since $\sigma$ and $\tau$ are nice names, they are of the form 
$$\sigma=\bigcup_{\beta<\alpha}\{\check\beta\}\times X_\beta\text{ and }\tau=\bigcup_{\beta<\alpha}\{\check\beta\}\times Y_\beta,$$
where $\langle X_\beta\mid\beta<\alpha\rangle,\langle Y_\beta\mid\beta<\alpha\rangle\in M$. 
Let $G$ be $\PP$-generic over $\MM$ and let $\beta_0\in\alpha$ be minimal such that 
$\beta_0\in\sigma^G\cap\tau^G$. Put $\mu=\{\langle\check\beta_0,\one_\PP\rangle\}\cup\{\langle\mu_\beta,p\rangle\mid\beta\in(\beta_0,\alpha),p\in X_\beta\}$,
where $\mu_\beta=\check\beta_0\cup\bigcup_{\gamma<\beta}\{\check\gamma\}\times Y_\beta$ for every $\beta\in(\beta_0,\alpha)$. 
Then $$(\mu_\beta)^G=\begin{cases}
                       \beta,&\beta\in\tau^G\\
                       \beta_0,&\text{otherwise}.
                      \end{cases}$$
Clearly, $\mu^G=\sigma^G\cap\tau^G$ is as desired.
\end{proof}

\begin{theorem}\label{thm:pretame nice}
 Let $\MM=\langle M,\C\rangle$ be a countable transitive model of $\KM$. Then a notion of class forcing $\PP$ for $\MM$ is pretame if and only if it is densely nice.
\end{theorem}

\begin{proof}
 \setcounter{claim}{0}
Suppose first that $\PP$ is pretame. It is straightforward to check that whenever there is a dense embedding $\pi\colon\PP\to\QQ$ in $\C$ for some notion of class forcing $\QQ$ for $\MM$, then $\QQ$ is also pretame. Then by Example \ref{example:nice} (2), every such $\QQ$ is nice. 

Conversely, suppose that $\PP$ is not pretame. Since $\PP$ satisfies the forcing theorem over $\MM$ (because every notion of class forcing does so over a model of $\KM$ -- see either \cite[Lemma 15]{antos2015class} or \cite[Corollary 5.8]{ClassForcing}), we may, without loss of 
generality, assume that $\PP=\langle P,\leq_\PP\rangle$ is an $M$-complete Boolean algebra, and we may also assume that $P=\On^M$. We will extend $\PP$ to a notion of class forcing $\QQ$ for $\MM$ which is 
not nice and so that $\PP$ is a dense subforcing of $\QQ$. By Lemma \ref{lemma:pretame cofinal function} there are a class name $\dot F\in\C^\PP$, $\kappa\in\On$ and $p\in\PP$ such that $p\Vdash_\PP\anf{\dot F:\check\kappa\ra\On\text{ is cofinal}}$. 
For the sake of simplicity, suppose that $p=\one_\PP$. 

For every $\alpha,\beta<\kappa$ and $p\in\PP$, let $$X_{p,\alpha,\beta}=\{\langle\gamma,\delta\rangle\in\On^M\times\On^M\mid\exists q\leq_{\PP}p\,[q\Vdash_\PP\dot F(\check\alpha)=\check\gamma\wedge\dot F(\check\beta)=\check\delta]\},$$
and let
$$Y_{p,\alpha}=\{\gamma\in\On^M\mid\exists q\leq_\PP p\,[q\Vdash_\PP\dot F(\check\alpha)=\check\gamma]\}.$$ 

\begin{claim}\label{claim:alpha}
For each $p\in\PP$ there is $\alpha<\kappa$ such that for all $\beta<\kappa$, $X_{p,\alpha,\beta}$ is a proper class.  
\end{claim}

\begin{proof}
Suppose the contrary. Then for every $\alpha<\kappa$ there exists $\beta_\alpha<\kappa$ such that $X_{p,\alpha,\beta_\alpha}$ is
set-sized. In particular, this implies that for every $\alpha<\kappa$, $\{\gamma\in\On^M\mid\exists q\leq_{\PP} p\,[q\Vdash_{\PP}\dot F(\check\alpha)=\check\gamma]\}$
is set-sized. 
But then $p$ forces that the range of $\dot F$ is bounded in the ordinals, a contradiction. 
\end{proof}

Let $C=\langle C_i\mid i\in\On^M\rangle\in\C$ be an enumeration of subclasses of $\On^M\times\On^M$ such that each $C_i$ is of the form
$X_{p,\alpha,\beta}$ for some $p\in\PP$ and $\alpha,\beta<\kappa$ such that $X_{p,\alpha,\beta}$ is a proper class, 
and moreover each $X_{p,\alpha,\beta}$ which is a proper class appears unboundedly often in the enumeration $C$. 

We will next perform a recursive construction to build two classes $D,E\in\C$, in a way that in particular each of $D\cap E$, $D\setminus E$ and $E\setminus D$ has a proper class sized intersection with $Y_{p,\alpha}=\{\gamma\mid\langle\gamma,\gamma\rangle\in X_{p,\alpha,\alpha}\}$ whenever $Y_{p,\alpha}$ is a proper class. 
The construction of the classes $D,E$ will satisfy further properties which will be used in the proof of Claim \ref{claim:Z_q,alpha^p} below. 

Let $D_0=D_0'=E_0=E_0'=\emptyset$. Suppose that $D_i,D_i',E_i,E_i'$ have already been defined such that $D_i\cap D_i'=E_i\cap E_i'=D_i'\cap E_i'=\emptyset$ 
and $D_i\cup D'_i=E_i\cup E'_i$. 
We define $F_i=D_i\cup D_i'= E_i\cup E_i'$. 
Let $\langle\gamma_0,\delta_0\rangle,\langle\gamma_1,\delta_1\rangle,\langle\gamma_2,\delta_2\rangle$ be the lexicographically least pairs of ordinals in $C_i$ such that each pair $\langle\gamma_k,\delta_k\rangle$ contains at least one ordinal not in $F_i\cup\{\gamma_j\mid j<k\}\cup\{\delta_j\mid j<k\}$, and 
$\gamma_0,\delta_0$ additionally satisfy (if possible)
\setcounter{equation}{0}
\begin{align}
 &\gamma_0\notin F_i\wedge\delta_0\notin D_i',
\end{align}
and $\gamma_1,\delta_1$ satisfy in addition (if such exist)
\begin{align}
 &\gamma_1\notin F_i\cup\{\gamma_0,\delta_0\}\wedge\delta_1\notin E_i'.
\end{align}

In the successor step, we will enlarge $D_i,D_i',E_i$ and $E_i'$ to $D_{i+1},D_{i+1}',E_{i+1}$ and $E_{i+1}'$ by
putting distinct ordinals, which are not in $F_i$, into the sets $D_{i+1}\cap E_{i+1}$, $D_{i+1}\cap E_{i+1}'$ and $D_{i+1}'\cap E_{i+1}$. 
First, we put each ordinal in $\{\gamma_0,\delta_0\}$
which is not in $F_i$ into $D_{i+1}\cap E_{i+1}'$. Next, we put all ordinals amongst $\{\gamma_1,\delta_1\}$
that are not in $F_i\cup\{\gamma_0,\delta_0\}$ into $D_{i+1}'\cap E_{i+1}$. 
Finally, we put every ordinal in $\{\gamma_2,\delta_2\}$
which is not yet in $F_i\cup\{\gamma_0,\gamma_1,\delta_0,\delta_1\}$ into $D_{i+1}\cap E_{i+1}$. 
Note that by construction, $D_{i+1}\cap D'_{i+1}=E_{i+1}\cap E_{i+1}'=D_{i+1}'\cap E_{i+1}'=\emptyset$ and
$D_{i+1}\cup D'_{i+1}=E_{i+1}\cup E'_{i+1}$. 

At limit stages, we take unions, e.g.\ if $j$ is a limit ordinal, we let $D_j=\bigcup_{i<j}D_i$. 
Finally, let $D=\bigcup_{i\in\On^M}D_i\in\C$ and let $E=\bigcup_{i\in\On^M}E_i\in\C$.

Note that at each stage $i$ such that $C_i=X_{p,\alpha,\alpha}$ for some $p\in\PP$ and $\alpha\in\On^M$, each of the classes 
$D\cap E$, $D\setminus E$ and $E\setminus D$ obtains a new element from $Y_{p,\alpha}$. 
Since there are class many such stages, each of $D\cap E$, $D\setminus E$ and $E\setminus D$ has a proper class-sized intersection with $Y_{p,\alpha}$ whenever $Y_{p,\alpha}$ is a proper class.

Let $a=\{\alpha<\kappa\mid\exists p\in\PP\,[p\Vdash_\PP\dot F(\check\alpha)\in\check D]\}$ and let
$b=\{\alpha<\kappa\mid\exists p\in\PP\,[p\Vdash_\PP\dot F(\check\alpha)\in\check E]\}$. 
We extend $\PP$ to a forcing notion $\QQ$ by 
adding suprema for each of the classes
\begin{align*}
 R_{\alpha}&=\{p\in\PP\mid p\Vdash_\PP\dot F(\check\alpha)\in\check D\}\textrm{ and}\\
 S_\beta&=\{p\in\PP\mid p\Vdash_\PP\dot F(\check\beta)\in\check E\}
\end{align*}
for $\alpha\in a$ and $\beta\in b$, as described in Section \ref{sec:intro}. Let $p_\alpha=\sup_\QQ R_\alpha$ and let $q_\beta=\sup_\QQ S_\beta$ for 
$\alpha\in a$ resp. $\beta\in b$. Since $\MM$ is a model of $\KM$, $\QQ$ satisfies the forcing theorem.

We will show that $\QQ$ is not nice. Let $\dot G$ denote the canonical class name for the $\QQ$-generic filter. Consider the $\QQ$-names
$$\sigma=\{\langle\check\alpha,p_\alpha\rangle\mid\alpha\in a\}\text{ and }\tau=\{\langle\check\alpha,q_\alpha\rangle\mid\alpha\in b\}$$
for $\{\alpha<\kappa\mid\dot F^{\dot G}(\alpha)\in\check{D}\}$ and $\{\alpha<\kappa\mid\dot F^{\dot G}(\alpha)\in\check{E}\}$ respectively.
By Lemma \ref{lemma:intersection}, for every $\QQ$-generic filter $G$ there is a $\QQ$-name $\mu$ such that $\mu^G=\sigma^G\cap\tau^G$.  
We claim that $M^\QQ$ contains no nice name for $\sigma^G\cap\tau^G$. Suppose for a contradiction that there are $p\in\QQ$
and a nice name $\nu\in M^\QQ$ such that $p\Vdash_\QQ\nu=\sigma\cap\tau$. By density of $\PP$ in $\QQ$, we may assume that $p\in\PP$.
Since $\nu$ is a nice name, it is of the form 
$$\nu=\bigcup_{\alpha<\kappa}\{\check\alpha\}\times A_\alpha,$$
where each $A_\alpha\subseteq\QQ$ is a set-sized antichain in $M$. 

Let $\alpha<\kappa$ be as in Claim \ref{claim:alpha}. We may assume that $A_\alpha$ only contains conditions
which are compatible with $p$. 

\begin{claim}\label{claim:Z_q,alpha^p}
For every $q\in A_\alpha$, 
$$Z_q=\{\gamma\in\On^M\mid\exists r\in\PP\,[r\leq_\QQ p,q \text{ and } r\Vdash_\PP\dot F(\check\alpha)=\check\gamma]\}$$
is a set in $M$.  
\end{claim}

\begin{proof}
We first consider $q\in A_\alpha\cap\PP$. 
By assumption $p$ and $q$ are compatible, and since $\PP$ is a Boolean algebra, $Z_q=Y_{p\wedge q,\alpha}$.
Assume for a contradiction that $Y_{p\wedge q,\alpha}$ is a proper class. Then by our construction, $Y_{p\wedge q,\alpha}\setminus D$ is a proper class as well.
Take $\gamma\in Y_{p\wedge q,\alpha}\setminus D$ and $r\leq_\PP p\wedge q$ with 
$r\Vdash_\PP\dot F(\check\alpha)=\check\gamma$. Let $G$ be $\QQ$-generic with $r\in G$. Then 
$p,q\in G$ and so $\alpha\in\nu^G=\sigma^G\cap\tau^G$. On the other hand, since $\dot F^G(\alpha)=\gamma\notin D$, we have $p_\alpha\notin G$ 
and hence $\alpha\notin\sigma^G$. This is a contradiction. 

Next, suppose that $q=p_\alpha$ and assume for a contradiction that $Z_{p_\alpha}$ is a proper class. Then $Y_{p,\alpha}$ is a proper class,
so $Y_{p,\alpha}\cap (D\setminus E)$ is also a proper class. Now let $r\leq_\PP p$ and $\gamma\in D\setminus E$ be such that 
$r\Vdash_\PP\dot F(\check\alpha)=\check\gamma$. Then $r\leq_\QQ p_\alpha$ by the definition of $p_\alpha$. If $G$ is $\QQ$-generic with $r\in G$, 
then $\alpha\in\nu^G$. Since $\gamma\notin E$, we have $\alpha\notin\tau^G$. This is a contradiction.

Next, suppose that $q=p_\beta\in A_\alpha$ for some $\beta\neq\alpha$. If there is some $\langle\gamma,\delta\rangle\in X_{p,\alpha,\beta}$
such that $\delta\in D$ but $\gamma\notin D\cap E$, then take $r\leq_\PP p$ such that $r\Vdash_\PP\dot F(\check\alpha)=\check\gamma\wedge\dot F(\check\beta)=\check\delta$
and a $\QQ$-generic filter containig $r$. Since $\delta\in D$ we have $p_\beta\in G$ and so $\alpha\in\nu^G$. On the other 
hand, $\dot F^G(\alpha)=\gamma\notin D\cap E$, so $\alpha\notin\sigma^G\cap\tau^G$. So there can be no such $\langle\gamma,\delta\rangle\in X_{p,\alpha,\beta}$.
Hence for all $\langle\gamma,\delta\rangle\in X_{p,\alpha,\beta}$, if $\delta\in D$ then $\gamma\in D\cap E$. 
Suppose for a contradiction that $Z_{p_\beta}$ is a proper class. Consider now the first stage $i$ such that $X_{p,\alpha,\beta}=C_i$.
Since $Y_{p,\alpha}$ is a proper class, there is $\langle\gamma,\delta\rangle\in X_{p,\alpha,\beta}$ such that $\gamma\notin F_i$.
If there is such a pair which additionally satisfies that $\delta\notin D_i'$, then we are in case (1)
in the recursive construction of $D$ and $E$ and so this would imply that $\gamma$ ends up in $D\setminus E$ and $\delta\in D$. 
But we have already shown that this is impossible. So for every pair $\langle\gamma,\delta\rangle\in X_{p,\alpha,\beta}$ with $\gamma\notin F_i$
we have $\delta\in D_i'$. In particular, if $\delta\in D$ then $\gamma\in F_i$. 
This implies that $Z_{p_\beta}\subseteq F_i$: Let $\gamma\in Z_{p,\beta}$. Then there is $\delta\in\On^M$ such that $\langle\gamma,\delta\rangle\in X_{p,\alpha,\beta}$. By the definition of $p_\beta$ we have $\beta\in D$ and hence by our previous argument $\gamma\in F_i$.
But then $Z_{p_\beta}$ is not a proper class, which is a contradiction. 

The case $q=q_\alpha$ is analogous to the case $q=p_\alpha$. 
Finally, suppose that $q=q_\beta$ for some $\beta\neq\alpha$. As in the previous case $q=p_\beta$, we can conclude that for all $\langle\gamma,\delta\rangle\in X_{p,\alpha,\beta}$, if $\delta\in E$ then $\gamma\in D\cap E$. As above, we assume that $Z_{q_\beta}$ is not in $M$ and we let $i$ be the least ordinal such that $C_i=X_{p,\alpha,\beta}$. After choosing $\gamma_0,\delta_0$ in the recursive construction of $D$ and $E$, there is still a pair $\langle\gamma_1,\delta_1\rangle$ such that $\gamma_1\notin F_i^+=F_i\cup\{\gamma_0,\delta_0\}$, since $Y_{p,\alpha}$ is a proper class. If possible, this pair is chosen such that $\delta_1\notin E_i'$. But then $\gamma_1$ is put into $E\setminus D$ and $\delta_1$ ends up in $E$. However, we have already argued that this cannot occur. But then for every such pair $\langle\gamma_1,\delta_1\rangle\in X_{p,\alpha,\beta}$ with $\gamma_1\notin F_i^+$, we have $\delta_1\in E_i'$, and so $Z_{q_\beta}$ is contained in the set $F_i^+$, which is a contradiction.
\end{proof}

By Claim \ref{claim:Z_q,alpha^p} and since $A_\alpha\in M$, we have that 
$$B=\bigcup_{q\in A_\alpha}Z_q\in M.$$
Since $Y_{p,\alpha}$ is a proper class, so is $Y_{p,\alpha}\cap D\cap E$ by our construction, and 
hence there must be some $\gamma\in (Y_{p,\alpha}\cap D\cap E)\setminus B$. Let now $q\leq_\PP p$ such that 
$q\Vdash_\PP\dot F(\check\alpha)=\check\gamma$ and take a $\QQ$-generic filter $G$ with $q\in G$. 
Then $\dot F^G(\alpha)=\gamma\in D\cap E$, so $\alpha\in\sigma^G\cap\tau^G$. Therefore there is some $r\in A_\alpha\cap G$. 
Take $s\in G$ with $s\leq_\QQ q,r$. Then $s\Vdash_\QQ\check\gamma=\dot F(\check\alpha)\in\check B$, contradicting the choice of $\gamma$. 
\end{proof}

\begin{theorem}\label{thm:ord-cc very nice}
Suppose that $\MM=\langle M,\C\rangle$ is a countable transitive model of $\GBC^-$. 
 A notion of class forcing $\PP$ for $\MM$ which satisfies the forcing theorem satisfies the $\On$-cc if and only if it is densely very nice.
\end{theorem}

\begin{proof}\setcounter{claim}{0}
 Suppose first that $\PP$ satisfies the $\On$-cc and $\PP$ embeds densely into $\QQ$. It is easy to see that then $\QQ$ also satisfies the $\On$-cc and so by Example \ref{example:nice}, (3) it is very nice. 

Conversely, suppose that $\PP$ contains a class-sized antichain. We would like to extend $\PP$ via a dense embedding to a partial order which is not very nice. Since $\PP$ satisfies the forcing theorem, $\PP$ has a Boolean $M$-completion. As we are only interested in a dense property, we may therefore assume that $\PP$ is already an $M$-complete Boolean algebra. 

By the proof of Lemma \ref{lem:no sup}, we can find three disjoint subclasses of our given class-sized antichain, each 
of which contains a subclass which does not have a supremum in $\PP$. Denote these subclasses without suprema by $A,D$ and $E$, and let $B=A\cup D$ and $C=A\cup E$. 

\begin{claim}
 At least one of $\sup B$ and $\sup C$ does not exist in $\PP$.
\end{claim}

\begin{proof}
We show that if both $\sup B$ and $\sup C$ exist, then so does $\sup A$, contradicting our choice of $A$. Since $\PP$ is an $M$-complete Boolean algebra,
if $\sup B$ and $\sup C$ exist, then so does $p=\sup B\wedge\sup C$. We claim that $p$ is already the supremum of $A$. 
It is clear that every element of $A$ is below $p$. It remains to check that $A$ is predense below $p$. 
Let $q\leq_\PP p$. Since $B$ is predense below $q$, there are $r\leq_\PP q$ and $b\in B$ with $r\leq_\PP b$. Since $C$ is
predense below $r$, there are $s\leq_\PP r$ and $c\in C$ with $s\leq_\PP c$. In particular, $b$ and $c$ are compatible.
But they both belong to the antichain $B\cup C$, so $b=c\in B\cap C=A$. 
\end{proof}

Let $\QQ$ be the forcing notion obtained from $\PP$ by adding $\sup B$ and $\sup C$. 
By Lemma \ref{lemma:add suprema ft}, $\QQ$ satisfies the forcing theorem. Moreover, it follows from the separativity of $\PP$ that $\QQ$ is separative. 
We show that $\QQ$ is not very nice. Consider the $\QQ$-name 
$$\sigma=\{\langle\{\langle\check0,\sup B\rangle\},\sup C\rangle\}.$$ By definition, $\one_\QQ\Vdash_\QQ\sigma\subseteq\check2$. 

\begin{claim}
There is no nice $\QQ$-name $\tau$ such that $\one_\QQ\Vdash_\QQ\sigma=\tau$.
\end{claim}

\begin{proof}
Suppose for a contradiction that $\tau=\{\check0\}\times A_0\cup\{\check1\}\times A_1$, where $A_0,A_1\in M$ are antichains of $\QQ$, and 
$\one_\QQ\Vdash_\QQ\sigma=\tau$. Observe that $A_1\subseteq\PP$, since if for example $\sup B\in A_1$ and $G$ is $\QQ$-generic 
with $\sup B\in G$ and $\sup C\notin G$, then $1\in\tau^G\setminus\sigma^G$. The same works for $\sup C$. Therefore $\sup A_1$ exists in $\PP$.
We claim that $\sup A_1$ is the supremum of $A$.

Firstly, we show that every element 
of $A$ is below $\sup A_1$. Suppose for a contradiction that there is $a\in A$ with $a\nleq_\QQ\sup A_1$. Then by separativity
of $\QQ$ there is $p\leq_\PP a$ with $p\bot_\PP\sup A_1$. In particular, $p$ is incompatible with every element of $A_1$. 
Hence if $G$ is a $\QQ$-generic filter with $p\in G$ then $1\in\sigma^G\setminus\tau^G$, contradicting our assumptions on $\sigma$ and $\tau$. Secondly, we check that 
$A$ is predense below $\sup A_1$. Assume, towards a contradiction, that there is $p\leq_\PP\sup A_1$ with $p\bot_\PP a$ for each 
$a\in A$. Now $A_1$ is predense below $p$, so there exist $q\leq_\PP p$ and $r\in A_1$ with $q\leq_\PP r$. Again, this yields
that for any $\QQ$-generic filter $G$ with $q\in G$, $1\in\tau^G$ but $A\cap G=\emptyset$, so it is impossible that 
$\sup B$ and $\sup C$ are both in $G$. Hence $1\notin\sigma^G$, contradicting our assumptions on $\sigma$ and $\tau$.

We have thus shown that $\sup A$ exists in $\PP$, contradicting our choice of $A$.
\end{proof}
This proves that $\QQ$ is not very nice. 
\end{proof}

 The proof of Theorem \ref{thm:ord-cc very nice} actually shows that every notion of forcing $\PP$ which satisfies the
 forcing theorem but not the $\On$-cc, can be densely embedded into a notion of class forcing which satisfies the
 forcing theorem and is nice but not very nice. To see this, it remains to check that the partial order $\QQ$ constructed 
 above is nice. This follows from the following more general result:
 
\begin{lemma} Suppose that $\PP$ is a notion of class forcing which satisfies the forcing theorem. 
If $\PP$ is nice and $\QQ$ is obtained from $\PP$ by adding the supremum of some subclass $A\in\C$ of $\PP$, then 
$\QQ$ is also nice. 
\end{lemma}
\begin{proof}
Let $\sigma\in M^\QQ$ and $p\Vdash_\QQ\sigma\subseteq\check\gamma$ for some $p\in\QQ$ and $\gamma\in\On^M$. Let $\sigma^+$ denote the $\PP$-name obtained from $\sigma$ by replacing every ocurrence of $\sup A$ in $\tc(\sigma)$ by $\one_\PP$, and let $\sigma^-$ be defined recursively by $\sigma^-=\{\langle\tau^-,p\rangle\in\sigma\mid p\neq\sup A\}$. 
Let $q\leq_\QQ p$. Without loss of generality, we can assume that $q\in\PP$. If $q$ is incompatible with every element of $A$, 
then $q\Vdash_\QQ\sigma=\sigma^-$. But then there are $r\leq_\PP q$ and a nice $\PP$-name $\tau$ such that $r\Vdash_\PP\sigma^-=\tau$ and so
$r\Vdash_\QQ\sigma=\tau$. If there is some $a\in A$ such that $q$ is compatible with $a$, let $r\leq_\PP q,a$. Then 
$r\Vdash_\QQ\sigma=\sigma^+$ and so as in the previous case we can strengthen $r$ to some $s$ which witnesses that $\sigma^+$
has a nice $\PP$-name. 
\end{proof}

\begin{corollary} Suppose that $\MM=\langle M,\C\rangle$ is a countable transitive model of $\GBC^-$.
Every notion of class forcing which satisfies the forcing theorem but not the $\On$-cc is dense in a notion of 
class forcing which is nice but not very nice. \hfill\qedsymbol
\end{corollary}

\section{Open Questions}

Assume that $\MM$ is a countable transitive model of $\GB^-$ and that $\PP$ is a notion of class forcing for $\MM$.
By Corollary \ref{cor:pretame axioms} and Theorem \ref{thm:pretame ft}, if $\PP$ preserves either Replacement or Collection, then $\PP$ satisfies the forcing theorem. By Theorem \ref{thm:sep rep}, if $\PP$ preserves Separation and satisfies the forcing theorem, then $\PP$ preserves Collection, and hence also Replacement. It would thus be very interesting to know the answer to the following question.

\begin{question}
   Assume that $\MM\models\GBC$ and that $\PP$ preserves Separation. Does this imply that $\PP$ satisfies the forcing theorem?
\end{question}

In Section \ref{sec:EMP}, we proved that over models of $\GBC^-$, pretameness is equivalent to the EMP for partial orders which satisfy the forcing theorem.

\begin{question}
 Assume $\PP$ is a notion of class forcing for $\MM$ which does not satisfy the forcing theorem. Does this imply that the EMP fails for $\PP$?
\end{question}

Most of our characterizations of pretameness use $\GBC^-$ as a base theory.

\begin{question}
  Is the assumption of a set-like wellorder of $M$ necessary for these results? Can it be replaced by the assumption of the existence of a hierarchy for some of them? Do any of these characterizations work over the base theory $\GB^-$?
\end{question}

In order to prove Theorem \ref{thm:pretame nice}, we work in $\KM$, since our proof uses that the forcing theorem is preserved when adding infinitely many suprema to a given notion of class forcing that satisfies the forcing theorem. $\KM$ implies that every notion of class forcing satisfies the forcing theorem. The following question arises. 

\begin{question}
  Can pretameness be characterized in terms of the existence of nice names for sets of ordinals in $\GBC$?
\end{question}

Every proof of the failure of the forcing theorem for first order definable notions of class forcing that is known to the authors (see \cite[Section 7]{ClassForcing} and Theorem \ref{thm:failure ft} of the present paper) uses the nonexistence of a first-order truth predicate in the ground model. This motivates the following question. 

\begin{question}
  Assume $\MM=\langle M,\C\rangle\models\GB^-$ such that $\C$ contains a first order truth predicate for $M$. Does every notion of class forcing that is (first-order) definable over $\langle M,\in\rangle$ satisfy the forcing theorem over $\MM$?
\end{question}

Let us emphasize that we ask for \emph{definable} notions of class forcing in the above. Let $\FF$ denote the notion of class forcing provided in \cite[Definition 2.4]{ClassForcing}, for which the forcing relation is not definable. It should be easy to define a version of $\FF$ relative to a first order truth predicate $T$, not only coding the $\in$-relation of the ground model by a binary relation on $\omega$, but also coding the truth predicate $T$ by a unary relation on $\omega$. If the forcing theorem held for that version of $\FF$, this would yield a truth predicate for the structure $\langle M,\in,T\rangle$ inside of $\C$, an assumption that is strictly stronger than asking for the existence of a first order truth predicate in $\C$. So if $\MM$ is as in the above question, then there is a notion of class forcing for $\MM$ that does not satisfy the forcing theorem, however it is not definable over $\langle M,\in\rangle$.

\bibliographystyle{alpha}
\bibliography{class-forcing}

\newcommand{\etalchar}[1]{$^{#1}$}
\begin{thebibliography}{HKL{\etalchar{+}}16}

\bibitem[Ant15]{antos2015class}
Carolin Antos.
\newblock {Class-Forcing in Class Theory}.
\newblock {\em arXiv preprint}, arXiv:1503.00116, 2015.

\bibitem[BJW82]{MR645538}
A.~Beller, R.~Jensen, and P.~Welch.
\newblock {\em {Coding the universe}}, volume~47 of {\em {London Mathematical
  Society Lecture Note Series}}.
\newblock Cambridge University Press, Cambridge-New York, 1982.

\bibitem[Fri00]{MR1780138}
Sy~D. Friedman.
\newblock {\em {Fine structure and class forcing}}, volume~3 of {\em {de
  Gruyter Series in Logic and its Applications}}.
\newblock Walter de Gruyter \& Co., Berlin, 2000.

\bibitem[Har77]{MR0465866}
Leo Harrington.
\newblock Long projective wellorderings.
\newblock {\em Ann. Math. Logic}, 12(1):1--24, 1977.

\bibitem[HKL{\etalchar{+}}16]{ClassForcing}
Peter Holy, Regula Krapf, Philipp L\"ucke, Ana Njegomir, and Philipp Schlicht.
\newblock Class forcing, the forcing theorem and boolean completions.
\newblock {\em J. Symb. Log.}, 81(4):1500--1530, 2016.

\bibitem[Kra16]{thesis}
Regula Krapf.
\newblock {\em {Class forcing and second-order arithmetic}}.
\newblock PhD thesis, University of Bonn, 2016.

\bibitem[Kun80]{MR597342}
Kenneth Kunen.
\newblock {\em Set theory}, volume 102 of {\em Studies in Logic and the
  Foundations of Mathematics}.
\newblock North-Holland Publishing Co., Amsterdam-New York, 1980.
\newblock An introduction to independence proofs.

\bibitem[Sta]{stanley_classforcing}
Maurice~C. Stanley.
\newblock Class forcing.
\newblock Incomplete draft, 1996, unpublished.

\bibitem[Sta84]{stanley_thesis}
Maurice~C. Stanley.
\newblock {\em A unique generic real}.
\newblock PhD thesis, U.\ C.\ Berkeley, 1984.

\bibitem[Zar73]{MR0345819}
Andrzej Zarach.
\newblock {Forcing with proper classes}.
\newblock {\em Fund. Math.}, 81(1):1--27, 1973.
\newblock Collection of articles dedicated to Andrzej Mostowski on the occasion
  of his sixtieth birthday, I.

\end{thebibliography}
  
\end{document}